\documentclass[12pt]{article}

\usepackage{url}
\usepackage{array,longtable}\setlongtables
\usepackage{amsfonts}
\usepackage{amsmath}
\usepackage{amsthm}
\usepackage{verbatim,graphicx}

{\catcode `\@=11 \global\let\AddToReset=\@addtoreset}
\AddToReset{equation}{section}

\newtheorem{lemma}{Lemma}[section]
\newtheorem{theorem}[lemma]{Theorem}
\newtheorem{proposition}[lemma]{Proposition}

\newtheorem{corollary}[lemma]{Corollary}
\theoremstyle{definition}
\newtheorem{definition}[lemma]{Definition}
\theoremstyle{definition}
\newtheorem{remark}[lemma]{Remark}

\newcommand{\R}{\mathbb{R}}

\usepackage{color}
\definecolor{darkred}{rgb}{0.8,0,0}

\def\hat{\widehat}
\def\tilde{\widetilde}

\usepackage{psfrag}
\newcommand{\observation}[1]{}
\newcommand{\generalizations}[1]{}

\newcommand{\hiddenfootnote}[1]{}

\newcommand{\cC}{{\cal C}}

\def\hat{\widehat}
\def\tilde{\widetilde}
\def \bfo {\begin {eqnarray*} }
\def \efo {\end {eqnarray*} }
\def \ba {\begin {eqnarray*} }
\def \ea {\end {eqnarray*} }
\def \beq {\begin {eqnarray}}
\def \eeq {\end {eqnarray}}
\def \supp {\hbox{supp}\,}

\def \dist {\hbox{dist}}

\def \p {\partial}

\begin{document}

\title{\Large {\bf}
Stability of the unique continuation for the wave operator via Tataru inequality and applications
}
\author{Roberta Bosi, Yaroslav Kurylev, Matti Lassas
}
\maketitle
\begin{abstract} In this paper we study the stability of the unique continuation in the case of the wave equation with variable coefficients independent of time. We prove a logarithmic estimate in a arbitrary domain of $\R^{n+1}$, where all the parameters are calculated explicitly in terms of the $C^1$-norm of the coefficients and on the other geometric properties of the problem. We use the Carleman-type estimate proved by Tataru in 1995 and an iteration of the local stability.
We apply the result to the case of a wave equation with data on a cylinder an we get a stable estimate for any positive time, also after the first conjugate point associated with the geodesics of the metric of the variable coefficients.
\end{abstract}

\paragraph{Keywords:}
wave equation, unique continuation property, stability,  analysis on manifolds, optimal control time.
\paragraph{Mathematical Subject Classification 2010.}
Primary: 35B53, 35B60, 35L05, 58J45, 93B05.\; Secondary: 35R30, 31B20, 58E25.

\section{Introduction}
We consider the wave operator in $\R^{n+1}$,
\begin{eqnarray}\label{wave_op}
P(y,D) = -D_0^2+ \sum_{j,k=1}^n g^{jk}(x)  D_j D_k +  \sum_{j=1}^n h^j(x) D_j + q(x),
\end{eqnarray}
where  $y=(t,x)\in \R \times \R^{n}$ are the time-space variables, $D_0 = -i \partial_{t}$, $D_j = -i \partial_{x_j}$.
The  coefficients $g^{jk} \in C^1(\R^n)$ are real and independent of time, and  $[g^{jk}]$ is a symmetric positive-definite matrix.
The  coefficients $h^j, q\in C^{0}(\R^n)$
are complex valued and independent of time.
\\
An operator  $P(y,D)$ is said to have the unique continuation property if
 for any solution $u$ to $Pu=0$ in a connected open set $\Omega \subset \R^{n+1}$ and vanishing on an open subset $B\subset \Omega$, it follows that u vanishes in $\Omega$.\\
 In the paper \cite{T1} Tataru  proved for the first time the unique continuation property for \eqref{wave_op} across every non-characteristic $C^2$-hypersurface with no limitation to the normal direction.
The key point of these results is a Carleman-type estimate involving an exponential pseudo-differential operator.
\\
Much is known about the consequences of the general unique continuation property for the corresponding Cauchy problem.
Actually  the unique continuation property has proved to be instructive in many areas of mathematics, e.g. in studying the uniqueness for linear and nonlinear PDEs together with their blow up or traveling wave solutions \cite{E}, in studying  the Anderson localization  \cite{BK},
in control theory to get controllability results \cite{T3,T4}, in inverse problems to obtain uniqueness and stability estimates \cite{I}.
In particular Tataru's result \cite{T1} is crucial for the development of the Boundary Control method (see \cite{Be} for pioneering work and \cite{KKL} for detailed exposition of the further developments).
\\
Concerning the continuous dependence of the unique continuation property, that is its stability, less results are available. The elliptic and the parabolic cases have been studied in several settings by using either Carleman estimates or some versions of the three ball theorem (see \cite{A}, for a review of the results).
\\
To our knowledge the hyperbolic case like \eqref{wave_op} is still open for arbitrary domains and arbitrary matrix valued coefficients $g^{jk}(x)$, while there exist results for particular coefficients or domains (see \cite{R,V}).
This is maybe related to the difficulty of using the standard Carleman estimates for hyperbolic operators in order to prove the unique continuation close to the characteristic directions, that is the reason why Tataru's work was so important in this field.
\\
The aim of the present work is then to prove  a global stability estimate for the unique continuation of
the operator $P(y,D)$.\\
In a previous work \cite{B} we proved this property for the local case.
Namely, given $S=\{y\in \Omega;\psi(y)=0\}$  a $C^{2,\rho}$-smooth oriented hypersurface, which is non-characteristic
in $\Omega$, for some fixed $\rho\in(0,1)$,
we assume that $u \in H^1(\Omega)$ is supported in $\{y; \psi(y) \le  0\}\cap \Omega$, and $P(y,D)u \in L^2(\Omega)$.
Then, for each $y_0 \in S$,  with $\psi'(y_0)\neq 0$,
we find
$R,r$ with $R \ge 2r >0$
such that the following stability estimate holds:
\begin{eqnarray*}
\|u\|_{L^{2}(B(y_0,r))} \le c_{111} \frac{\|u\|_{H^1(B(y_0,2R))}}{\ln \Big(1 + \frac{\|u\|_{H^1(B(y_0,2R))}}{\|Pu\|_{L^2(B(y_0,2R))}}\Big)}\,.
\end{eqnarray*}
Here $B(y_0,r)$ is a ball in $\R^{n+1}$ of radius $r>0$ centered in $y_0$ and $B(y_0,r)\subset B(y_0,2R)\subset \Omega$.
The radii $r$ and $R$ and the coefficient $c_{111}$ have been explicitly calculated with dependency on the geometric parameters and on the function $\psi$ in \cite{B}.
\\
In this work we use the previous local stability inequality to prove a similar logarithmic estimate
for quite general domains of $\R^{n+1}$.\\
Moreover we propose a procedure to calculate all the constants involved, dependent on the norms of $\psi$, the coefficients in \eqref{wave_op}, the properties of the domains and the smooth localizers. The procedure is described in Proposition \ref{prop} and Appendix A.
\\
Concerning the proof, in the unpublished manuscript \cite{T}, Tataru suggested the possibility of obtaining a log-stability result,
by splitting the estimate for high and low temporal frequencies and by using Gevrey-class localizers to improve the estimates of $u$ for low temporal frequencies. Here and in \cite{B} we have advanced that idea,
by employing tools of subharmonic functions and proper choice of the localizers in the iterating procedure, together with the explicit computations of the uniform radii $r,R$ and the time frequencies used in the iteration.
Of fundamental importance is the calculation of the positive lower bound of the radius $r$, without which the iterative procedure could stop before covering the desired subdomain of $\Omega$.\\
The technique used consists in iterating the local stability result, but considering the low temporal frequencies separately from the  high temporal frequencies.
The advantage is that one can avoid the usual $(\ln\ln... |\ln \|Pu\|_{L^2}|)^{-\theta}$ iterated estimate (for $\|u\|_{H^1}=1$ and $\|Pu\|_{L^2} << 1$, $\theta \in (0,1)$) and get a $(|\ln \|Pu\|_{L^2}|)^{-\theta}$ results.
As a consequence one obtains a stable control of the solution $u$ inside $\Omega$, for any positive time. Moreover, we can come as close as we wish to the optimal time of the control $T_{opt}$, i.e. the time to reach the uniqueness in  $\Omega$ (see Corollary \ref{cor2}, as example of computation).
The importance of this issue has also been underlined in \cite{R}, who worked with FBI transform technique to get a log-stability estimate for large times. Hoermander in \cite{H0} proved an upper bound of the type $\sqrt{27/23}T_{opt}$. The issue of reaching $T_{opt}$ for \eqref{wave_op} has been solved in \cite{T1}, see also \cite{H,RZ}. Here we can derive the stable determination of it.
\\
Like in the elliptic case, many possible applications can be derived out of it.
In particular we plan to use the inequalities in Theorem \ref{global}-\ref{global2} to obtain an explicit modulus
of continuity for the
inverse problem for the wave operator on manifolds. This would improve  the existing inverse
stability results for Riemannian  manifolds, which are currently based either on compactness-type arguments, see \cite{AKKLT, KLY}, or on very strong geometrical conditions for the coefficients, e.g.\ in \cite{EINT,IIY,IIY2}.
Here is important to be able to relate the explicit estimates with some geometric invariant of the manifold (Ricci curvature, sectional curvature, diameter, etc.).
\\
As application, in section \ref{sec3} we apply Theorem \ref{global2} to the case of an arbitrary domain of influence in $\R^{n+1}$.
This is a special case of manifold, once one considers $g^{jk}$ as the inverse of the metric tensor.
We start with a time-cylinder where the wave solution vanishes (or has small data) and we get the stability in any compact subset of the associated domain of influence at time $T$.
The control of solution in a stable way in the domain of influence can have numerous important applications in inverse problems and in in control theory.
Here we consider also the case in which
the ray field has also singularities, i.e. behind the corresponding cut-locus. This means that in principle we are able to deal with manifolds that possesses conjugate points, trapped rays and other singularities of geodesics.
Thus, we remove the usual non-trapping conditions used in the Carleman estimates.
\\
The paper is organized as follows: in Section \ref{sec2} we prove Theorems \ref{global} and \ref{global2}, in Section \ref{sec3} we present the application to the case of a domain of influence of the wave solution vanishing in a small cylinder. In Appendix A we present the table with the estimates for the parameters used in Sec. \ref{sec2} and
we study the uniform estimates for the distance function $d_g$ and the related function $\psi$ defined in Sec. \ref{sec3}.
\\\\
We first introduce some assumptions.

\begin{paragraph}{\bf Assumption A1}
Let $\Omega$ be a connected open subset of $\R\times\R^n$.
Let $P(y,D)$ be the wave operator \eqref{wave_op}, with $g^{jk}(x) \in C^1(\Omega),$ $h^j, q \in C^{0}(\Omega)$.
We assume that $u \in H^1(\Omega)$
and that $P(y,D)u \in L^2(\Omega)$.
Assume that there is a function $\psi \in C^{2,\rho}(\Omega)$, for some  $\rho\in(0,1)$,  such that in a domain $\Omega_0 \subseteq \Omega$ one has $p(y,\psi'(y))\neq 0$ and $\psi'(y)\neq 0$, where $p(y,\xi)= -\xi_0^2 + g^{jk}(x) \xi_j\xi_k$ is the principal symbol of $P$.
\\
Assume that there exist values $\psi_{min} < \psi_{max}$ and a connected nonempty set $\Upsilon \subset \Omega_0$ such that:
$\supp(u)\cap \Upsilon = \emptyset$;\quad
and $\emptyset \neq  \{y\in \Omega_0; \psi(y) >  \psi_{max}\} \subset \Upsilon$\quad
(which implies that $\Omega_0$ contains a subset $\Upsilon$ where $u$ vanishes, and that the value $\psi_{max}$
is obtained for points inside the domain $\Omega_0$).\\
Assume that $\psi_{min}$ is such that the open set \;
$\Omega_a = \{y \in \Omega_0 \backslash \overline{\Upsilon}: \psi_{min} < \psi(y) < \psi_{max} \} $ is nonempty, connected and satisfies
$\dist(\partial \Omega_0, \Omega_a) >0$.
\end{paragraph}
\\\\
See remark \ref{remarks} for comments about the construction.

\begin{paragraph}{\bf Assumption A2}
We define $A(D_0)$ to be a pseudo-differential operator with symbol $a(\xi_0)$, $0\le a\le 1$, where
$a \in C^\infty_0(\R)$ is a smooth localizer supported in $|\xi_0|\le 2$ , equal to one in $|\xi_0|\le 1$.
Furthermore let $a\in G^{1/\alpha}_0(\R)$ for a fixed $\alpha \in(0,1)$.
Here $G^{1/\alpha}_0$ is the set of Gevrey functions of class $1/\alpha$ with compact support, defined in \cite{H1, R}. We also define the smooth localizer $b(y)$, supported in $|y|\le 2$, $0\le b\le 1$ and equal to one in $|y|\le 1$.
\end{paragraph}
\\\\
The main results of the paper are the following Theorems \ref{global} and \ref{global2}, together with their application in Section \ref{sec3} Theorem  \ref{global2 application}.
\\
\begin{theorem}\label{global}
Under the conditions of Assumption A1-A2,
define the open set $\Omega_1 = \Omega_0 \backslash \overline{\Upsilon}$ containing $\Omega_a$.
Then for every $0<\theta < 1$ we have
\begin{eqnarray*}
 \|u\|_{L^{2}(\Omega_a)} \le c_{160} \frac{\|u\|_{H^{1}(\Omega_1)}}{\Big(\ln\Big(1+\frac{\|u\|_{H^{1}(\Omega_1)}}{\|Pu\|_{L^{2}(\Omega_1)}}\Big)\Big)^{\theta}}.
\end{eqnarray*}
Moreover, for any $m \in (0,1]$ we get
\begin{eqnarray*}
\|u\|_{H^{1-m}(\Omega_a)} \le c_{160}^m \frac{\|u\|_{H^1(\Omega_1)}}{\Big(\ln \Big(1 + \frac{\|u\|_{H^1(\Omega_1)}}{\|Pu\|_{L^2(\Omega_1)}}\Big)\Big)^{m \theta}}\,.
\end{eqnarray*}
The constant $c_{160}$ is calculated in the proof.
\end{theorem}
The dependency of the constant $c_{160}$ from the geometric parameters of the problems and from $\psi$ and $\theta$ is described in Proposition \ref{prop}.

\begin{paragraph}{\bf Assumption A3}
Let $\Omega$ be a connected open subset of $\R\times\R^n$.
Let $P(y,D)$ be the wave operator \eqref{wave_op}, with $g^{jk}(x) \in C^1(\Omega),$ $h^j, q \in C^{0}(\Omega)$.
Let $u \in H^1(\Omega)$ and  $P(y,D)u \in L^2(\Omega)$.\\
In $\Omega$ we assume the existence of open connected subsets $\Lambda_k, \,\Omega_{0,k}$, a connected set $\Upsilon$ and functions $\psi_k$ for $k=1,2,\dots,K$ defined in this way:\\
1. $\psi_k  \in C^{2,\rho}(\Omega)$
for some $\rho\in(0,1)$,  such that $p(y,\psi_k'(y))\neq 0$ and $\psi_k'(y)\neq 0 $ in $\Omega_{0,k}$,  where $p(y,\xi)= -\xi_0^2 + g^{jk}(x) \xi_j\xi_k$ is the principal symbol of $P$.
\\
2.
Assume that:
$\supp(u)\cap \Upsilon = \emptyset$.
Define $\Upsilon_1 = \Omega_{0,1}\cap \Upsilon$ and for $k\ge 2$ set $\Upsilon_k =  \Omega_{0,k} \cap \Big(\bigcup_{j<k} \Lambda_j \cup \Upsilon \Big)$.
\\
Assume that there exist values $\psi_{min,k} < \psi_{max,k}$
such that:\quad
$\Big((\supp(u)\cap \Omega_{0,k})\setminus \bigcup_{j<k} \Lambda_j\Big) \subset \{y; \psi_k(y) \le  \psi_{max,k}\}$;\quad
and
$\emptyset \neq  \{y\in \Omega_{0,k}; \psi_k(y) >  \psi_{max,k}\} \subset \Upsilon_k$.
\\
3.
Assume that $\psi_{min,k}$ is such that\;
$\Lambda_k = \{y \in \Big(\Omega_{0,k}\setminus \overline{\Upsilon_k} \Big);\; \psi_{min,k} < \psi_k(y) < \psi_{max,k} \}$ is nonempty, connected and satisfies
$\dist(\partial \Omega_{0,k}, \Lambda_k) >0$.\\
4. Assume that $\Lambda=\cup_{k=1}^K \Lambda_k$ is a connected set.
\\
\end{paragraph}

\begin{theorem}\label{global2}
Under the conditions of Assumptions A2-A3,
define the open set
$\Omega_1 = \bigcup_{k=1}^K \Omega_{0,k} \backslash \overline{\Upsilon}$ containing $\Lambda$.
Then for every $0<\theta < 1$ we have
\begin{eqnarray*}
 \|u\|_{L^{2}(\Lambda)} \le c_{161} \frac{\|u\|_{H^{1}(\Omega_1)}}{\Big(\ln\Big(1+\frac{\|u\|_{H^{1}(\Omega_1)}}{\|Pu\|_{L^{2}(\Omega_1)}}\Big)\Big)^{\theta}}.
\end{eqnarray*}
Moreover, for any $m \in (0,1]$ we get
\begin{eqnarray*}
\|u\|_{H^{1-m}(\Lambda)} \le c_{161}^m \frac{\|u\|_{H^1(\Omega_1)}}{\Big(\ln \Big(1 + \frac{\|u\|_{H^1(\Omega_1)}}{\|Pu\|_{L^2(\Omega_1)}}\Big)\Big)^{m \theta}}\,.
\end{eqnarray*}
The constant $c_{161}$ is calculated in the proof.
\end{theorem}
The dependency of the constant $c_{161}$ from the geometric parameters of the problems and from $\psi$ and $\theta$ is described in Proposition \ref{prop}.

\section{Global Stability}\label{sec2}

\noindent {\bf Notations.} We start by introducing some notations and definitions used in the rest of the article:
first we consider $y=(t,x) \in \R \times \R^{n}$ the time-space variable and call $\xi=(\xi_0,\tilde\xi)$ its Fourier dual variable.
We remind that the exponential
pseudodifferential operator in Theorem \ref{th_carleman} is defined as
$e^{-\epsilon |D_0|^2/2\tau} v= \mathcal{F}^{-1}_{\xi_0 \to t}e^{-\epsilon \xi_0^2/2\tau}\mathcal{F}_{t'\to\xi_0}v$, with $\mathcal{F}$ and $\mathcal{F}^{-1}$ representing respectively the Fourier transform and its inverse. Then $e^{-\epsilon |D_0|^2/2\tau}$ is an integral operator with
kernel \quad $(\frac{\tau}{2\pi\epsilon})^{1/2}e^{-\tau|t'-t|^2/2 \epsilon }$.
We consider a pseudo-differential operator $A(D_0)$  with symbol $a(\xi_0) \in C^\infty_0(\R)$, $0\le a\le 1$,
supported in $|\xi_0|\le 2$ and equal to one in $|\xi_0|\le 1$.
Hence we can write $A(\beta |D_0|/\omega)v = \mathcal{F}^{-1}_{\xi_0 \to t}a(\beta |\xi_0|/\omega)\mathcal{F}_{t'\to\xi_0}v$ and the integral kernel is
 $(\frac{\omega}{2\pi\beta})^{1/2}{\hat a}(\frac{\omega|t'-t|}{\beta})$.
We will often work under the Assumption A2, where the symbol $a$ is of Gevrey class. The smooth localizer $b(y)$ is always supported in $|y|\le 2$ and equal to one in $|y|\le 1$.\\
The norm of the Sobolev space $H_{\tau}^s$ is defined as $\|u\|_{s,\tau} = \|(|\xi|^2+\tau^2)^{s/2}\mathcal{F}_{y \to \xi}u\|_{L^2}$, and the space $H^s$ corresponds to the case $\tau=1$.\\
According to our notations the positive coefficients denoted by $c_x$ with $x\ge 100$ are defined just once, independently on the variables $\mu, \tau$, and they are calculated explicitly in terms of the coefficients of the operator \eqref{wave_op} and the geometric parameters.  This is essential to finally recover the value of $c_{160}$ and the radii $R, r$ in Table \ref{tab1}.
\\
\\
\noindent
A first step is the following lemma, proven in \cite{B}, introducing a property often used in this section.
\begin{lemma}\label{lm_util}
Let $A(D_0)$ be a pseudo-differential operator with symbol $a(\xi_0)$, where
$a \in C^\infty_0(\R)$ is a smooth localizer supported in $|\xi_0|\le 2$ and equal one in $|\xi_0|\le 1$. Assume that
$ f(y) \in C^{\infty}_0(\R^{n}_x, G^{1/\alpha}_0(\R^1_t))$, where $0<\alpha<1$.
Then, for every $\mu > 0$, $\beta_1 > 2$, $v \in L^2_{loc}(\R^{n+1})$ there are two constants $c_{106},\, c_{107}$ independent of $\mu$ such that
\begin{eqnarray*}
a)&\|A(\beta_1 D_0/\mu)f(y)(1-A(D_0/\mu)) v\|_0 \le c_{107} e^{-c_{106} \mu^\alpha}\|v\|_0\,.
\end{eqnarray*}
Moreover, if $v \in L^2(\R^{n+1})$ and $h\in C^\infty_0(\R^{n+1})$ is a function such that $h \equiv 1$ on supp$(f)$, then
\begin{eqnarray*}
b)&\|A(\beta_1 D_0/\mu)f h v\|_0 \le \|f\|_\infty \|A(D_0/\mu)h(y) v\|_0 + c_{107} e^{-c_{106} \mu^\alpha}\|h v\|_0\, .
\end{eqnarray*}
If $v \in H^m_{loc}(\R^{n+1})$, $m\ge 1$,  then the estimate above holds also in $H^m(\R^{n+1})$,
\begin{eqnarray*}
c)&
\|A(\beta_1 D_0/\mu)f (1-A(D_0/\mu))v\|_m \le c_{108} e^{-c_{106} \mu^\alpha}\|v\|_m\,.
\end{eqnarray*}
\end{lemma}
\begin{proof}
See \cite{B} for the entire proof. Here we remind how to obtain the coefficients.\\
$a)$
On the set supp$[(1-a( \xi_0/\mu)) a(\beta_1 \xi^{1}_0/\mu)]$ one obtains $|\xi^{1}_0- \xi_0|^\alpha \ge (\mu - 2\mu/\beta_1)^{\alpha}$ and  the assumption $f(t,.) \in G^{1/\alpha}_0(\R_t)$ implies, uniformly in $x$ on a compact set $K\subset \R^n$ and
for some $c_3=c_3(\alpha, K)$, $c_{117}=c_{117}(\alpha,K)$ and $c_{106}=c_{117}(1 - 2/\beta_1)^{\alpha}/4$,
\begin{eqnarray}\label{fourier}
|\mathcal{F}_{t' \to (\xi^{1}_0- \xi_0)}[ f(t',x)]| \le c_3 e^{-c_{117}|\xi^{1}_0- \xi_0|^\alpha}
\le c_3 e^{-2c_{106} \mu^\alpha}e^{-c_{117}|\xi^{1}_0- \xi_0|^\alpha/2}.
\end{eqnarray}
The coefficient $c_3=c_3(\alpha, K)$ is proportional to $c_{1,f}$, the Gevrey parameter of $f$, that is \cite{H1,RD}
$$|D^\kappa (f(s))| \le c_{1,f}^{|\kappa|+1} (|\kappa|+1)^{|\kappa|/\alpha}, \; s \in \mbox{supp}(f).$$
We have $c_3 = c_{1,f} \mbox{Vol}(\supp(f))$ and $c_{117} = 1/(ec_3)^{\alpha}$.
We then estimate in the Fourier space the operator $A(\beta_1 D_0/\mu)f(\cdot) (1-A(D_0/\mu))$,
with $c_{107}=\big(c_3\frac{8}{\beta_1}\Gamma\Big(\frac{1}{\alpha}\Big)\frac{1}{\alpha (c_{117})^{1/\alpha}}\frac{1}{(\alpha c_{106})^{\frac{1}{\alpha-1}}}
)\big)^{1/2}$.
\begin{eqnarray*}
\|a(\beta_1 \xi^1_0/\mu)\mathcal{F}_{t'\to\xi_0^1}\Big(f(t',x) \big(\mathcal{F}^{-1}_{\xi_0\to t'}(1-a(\xi_0/\mu))\mathcal{F}_{t\to\xi_0}[v]\big)\Big)\|^2_0
\le c_{107}^2 e^{ -2c_{106}\mu^{\alpha} }\|v\|_0^2 \,.
\end{eqnarray*}
\end{proof}

\noindent
According to
the splitting $y=(t,x)$, the conormal bundle in $\R^{n+1}$ with respect to the foliation $x=$const is
defined as: \\
$N^*F := \{ (y,\xi) \in T^* \R^{n+1}; \mbox{ with } \xi=(\xi_0,\tilde \xi) \mbox{ and }\xi_0=0\}$.\\
Its reduction to a subset $K \subset \R^{n+1}$ is
\quad $\Gamma_{K} := \{ (y,\xi) \in T^*K, \, \xi_0=0 \}, $\\
and its fibre in $y_0$ is
\quad $\Gamma_{y_0} := \{ (y_0,\xi) \in N^*F \}.$
\\\\
We then recall the concept of  conormally strongly pseudoconvex function, alias  strongly pseudoconvex function with respect to $P$ on $\Gamma_{y_0}$ (\cite{T1,T2}).
\begin{definition}\label{def_pseudo_func}
A  $C^2$ real valued function  $\phi$  is {\it conormally strongly pseudoconvex}  with respect to $P$ at $y_0$ if
\begin{eqnarray}
\hspace{1mm}&&
Re\{\overline{p},\{p,\phi\}\}(y_0, \xi) > 0 
\\
\nonumber
\hspace{1mm}&&
\quad \mbox{on} \quad p(y_0,\xi)=0, \quad 0\neq \xi \in \Gamma_{y_0};
\\
\hspace{1mm}&&
\{\overline{p(y,\xi+i\tau\phi'(y))},p(y,\xi+i\tau\phi'(y))\}/(2i\tau) > 0 
\\
\nonumber
\hspace{1mm}&&
\quad \mbox{for} \; y=y_0,\; \quad 0\neq \xi \in \Gamma_{y_0}, \mbox{such that} \; p(y_0,\xi+i\tau\phi'(y_0))=0,  \, \tau>0.
\end{eqnarray}
\end{definition}

In particular, for the wave operator \eqref{wave_op} the conditions are void for non-characteristic surfaces $\phi=const$.
As consequence one can state the following Theorem (Theorem 2.1 in \cite{B}), where the Carleman-type estimate by Tataru is recalled.

\begin{theorem}\label{th_carleman} Let $\Omega$ be an open subset of $\R\times\R^n$.
Let $P(y,D)$ be the wave operator \eqref{wave_op}, with $g^{jk}(x) \in C^1(\Omega)$, $h^j, q \in C^{0}(\Omega)$.
Let $y_0 \in \Omega$ and $\psi \in C^{2,\rho}(\Omega)$ be real valued, for some fixed $\rho\in(0,1)$, such that $\psi'(y_0)\neq 0$ and $S=\{y;\psi(y)=0\}$ being an  oriented hypersurface non-characteristic
in $y_0$. \\
Consequently there is $\lambda>1$ such that $\phi(y)=\exp(\lambda \psi)$ is a conormally strongly pseudoconvex  function with respect to $P$ at $y_0$.
\\
Then there is a real valued quadratic polynomial $f$ defined in \eqref{def_f} with proper $\sigma > 0$,
and a ball $B_{R_2}(y_0)$ such that $f(y) < \phi(y)$ when $y \in B_{R_2}-\{y_0\}$ and $f(y_0)=\phi(y_0)$; and $f$ being a conormally strongly pseudoconvex function with respect to $P$ in $B_{R_2}$.
This implies that there exist $ \,\epsilon_0,\, \tau_0,\,c_{1,T},\,c_{2,T}, R$, such that, for each small enough  $\epsilon<\epsilon_0$ and large enough $\tau > \tau_0$, we have
\begin{eqnarray*}
 \|e^{-\epsilon |D_0|^2/2\tau} e^{\tau f} u\|_{1,\tau} \le c_{1,T}\, \tau^{-1/2}\|e^{-\epsilon |D_0|^2/2\tau} e^{\tau f} P(y,D)u\|_{0} + c_{2,T} e^{-\tau R_2^2/4\epsilon} \|e^{\tau f} u\|_{1,\tau}.
\end{eqnarray*}
Here $u \in H^1_{loc}(\Omega),$ with $P(y,D)u \in L^2(\Omega)$ and supp$(u) \subset B_{R}(y_0)$.
\end{theorem}

This last estimate was used in \cite{B} to prove local stability of the unique continuation with explicit coefficients.
We recall these results in the following proposition.

\begin{proposition}\label{prop}
Let $P$ the wave operator \eqref{wave_op}. Then, under the Assumption A1-A2, and using the result of Theorem \ref{th_carleman},
there exist two positive radii $R$ and $r$ such that the local stability results (i.e. Lemma \ref{lm_expo1} and Theorems 1.2 in \cite{B}) hold true in every point of $\Omega_0$, with the same parameters.\\
Moreover,  starting by the Assumptions A1, we are able to calculate all the constants involved in the local stability in a uniform way over $\Omega_0$.
The geometric parameters are constructed in Table \eqref{tab1} while the derived constants are  in the proof of  Theorems 1.1 and 1.2 of \cite{B}.\\
All the constants depend on :\\
- the coefficients in \eqref{wave_op} and their bounds:
\begin{eqnarray*}
|g^{jk}|_{C^1(\Omega_0)}, |h|_{C^0{(\Omega_0)}}, |q|_{C^0{(\Omega_0)}}, a_1, b_1,
\end{eqnarray*}
- the assumptions on the domains:
$$\mbox{dist}_{\mathbb{\R}^{n+1}}\{\partial \Omega_0, \Omega_a \} > 0,$$
- the non-characteristic condition and the non-vanishing condition upon $\psi'$:
\begin{eqnarray*}
p_1 = \min_{y \in \overline{\Omega}_0} p(y,\psi'), \,C_l= \min_{y \in \overline{\Omega}_0}{|\psi'(y)|},
\end{eqnarray*}
- the norms of $\psi$ (see (\ref{Appendix formula1}) for notations)
\begin{eqnarray*}
|\psi'|^2_{C^0(\Omega_0)}, |\psi''|^2_{C^0(\Omega_0)}, |\lambda  \psi|_{max, \Omega_0},
\end{eqnarray*}
- the norms of the smooth localizers, in time-space and frequency, together with their Gevrey parameters:
\begin{eqnarray*}
|\chi_1(s)|_{C^2(\R)}, |b(s)|_{C^2(\R)}, |a(\xi_0)|_{C^0(\R)}.
\end{eqnarray*}
\end{proposition}

\noindent We then need to reformulate Lemma 2.6
of \cite{B} in the case of more general assumptions.


\begin{lemma}\label{lm_expo}
Under the assumption A1,
let $y_0\in \Omega_0$ and $\varphi$ be the quadratic polynomial $\varphi(y)=f(y) -f(y_0)$, with $f$ defined in \eqref{def_f}.
Let $0<\alpha< 1$ and $\chi(s) \in G_0^{1/\alpha}(\R)$ be a localizer supported in $[-8\delta,\delta]$ and equal 1 in $[-7\delta,\delta/2]$. Let  $\mu >0,\, \delta>0, $ be given constants, $b \in C_0^{\infty}(\R^{n+1})$ and
$a \in C_0^\infty(\R)$.
Let $A(D_0)$ be a pseudodifferential operator with symbol $a$. If
\begin{eqnarray*}
u\in H^1(B_{2R}),\quad Pu \in L^2(B_{2R}), \quad
\|A(D_0/\mu)b((y-y_0)/R)Pu\|_0 \le c_A e^{-\mu^\alpha} \,,
\end{eqnarray*}
then for each $\tau \ge 0$, there are constants $c_{110}, c_{109}$ such that
\begin{eqnarray*}
\|e^{-\epsilon |D_0|^2/2\tau} e^{\tau \varphi} \chi(\varphi)b\big(\frac{y-y_0}{R}\big)P(y,D)u\|_0 \le \max\{c_{110},c_A\} e^{2\tau \delta-c_{109}\mu^{\alpha}}\,\max\{1,\, \|Pu\|_{L^2(B_{2R})}\}.
\end{eqnarray*}
\end{lemma}

Using Lemma \ref{lm_expo} we now reformulate Theorem 1.1. of \cite{B} with more general assumptions.

\begin{lemma}\label{lm_expo1}
Under the assumptions A1-A2, let $y_0 \in S=\{y;\psi(y)=0\}$
be an  $C^{2,\rho}$-oriented hypersurface, which is non-characteristic
in $y_0$ and
with $\psi'(y_0) \neq 0$. \\
We also assume that $u \in H^1(\Omega)$ is such that $\supp(u) \cap B_{2R}(y_0) \subset \{y; \psi(y) \le  0\}$.\\
Let
$b\in G^{1/\alpha}_0(\R^{n+1})$ be Gevrey functions with compact support, with $0<\alpha<1$.
Then,  for $\mu \ge 1$, if for some positive coefficients $c_U,\,c_P,\,c_A$
\begin{eqnarray*}
\|u\|_{H^1(B_{2R(y_0)})}\le c_U,\quad \|Pu\|_{L^2(B_{2R}(y_0))} < c_P, \quad \|A(D_0/\mu)b((y-y_0)/R)Pu\|_0 \le c_A e^{-\mu^\alpha},
\end{eqnarray*}
then,
there are constants $c_{150}, c_{131}, c_{132}$ independent of $\mu$, such that
\begin{eqnarray*}
\|A(D_0/\omega)b((y-y_0)/r)u\|_{H^1} \le c_{150} e^{-c_{132}\mu^{\alpha^2}},\quad \forall\, \omega \le \mu^\alpha/(3c_{131}).
\end{eqnarray*}
Moreover $c_{131}$ and $c_{132}$ are independent of $c_U,c_P,c_A$, while $c_{150}$ depends on them.\\
The dependency of all the constants is as described in Proposition \ref{prop}.
\end{lemma}
\begin{proof} The proof is identical to the one of Theorem 1.1. in \cite{B}.
Th. \ref{th_carleman} is used for the function $\chi(\varphi)b\big(\frac{y-y_0}{R}\big)u$ that is supported in $B_R(y_0) \cap \{y; \phi(y)\le \phi(y_0)\} \cap \{y; -8 \delta< \varphi(y) < \delta\}$. Moreover  $[P,\chi(\varphi)b\big(\frac{y-y_0}{R}\big)]u=[P,\chi(\varphi)\big)]b\big(\frac{y-y_0}{R}\big)u$ (since $D(\chi b) u = b D(\chi) u$), while in $B_R(y_0)$ one has $\chi(\varphi)b\big(\frac{y-y_0}{R}\big)u = \chi(\varphi)u$.
\\
Here we have just to recalculate the related coefficients, distinguishing the ones dependent upon the parameters $c_U, c_P, c_A$ from the ones independent of them.\\
We first list of coefficients independent of $c_U, c_P, c_A$, but dependent on the Gevrey parameters of the localizers and from the geometric constants $r,R,\delta$ (see Table \eqref{tab1}):
\begin{eqnarray*}
c_{1X}=c_{1X}(1/\alpha),\;
c_{2X} = 1/(eNc_{1X})^{\alpha},\;
c_{101}=c_{101}(\alpha),\;
c_{102}=c_{102}(\alpha,c_{101}),\\
c_{119}=\delta c_{1X}(\alpha),\;
B=\delta^{\alpha} c_{1X}(\alpha),\;
|\partial^k_t \varphi(y)|_{C^0(B_R)} \le c_{118}(R) (>1),\\
c_{114}=c_{1,T}^2 |g|^2_{C^1}|\chi_1|^2_{C^2}(1+|\varphi'|^4_{C^0}/\delta^4+|\varphi''|^2_{C^0}/\delta^2),\\
c_{115}=c_{2,T}^2(|\varphi'|^2_{C^0} +1)(3^3 e^{-3}/\delta^3)(1+|\chi_1'|^2_{C^0}/\delta^2),\\
c_{120}=\big(\frac{(1-\alpha)}{\alpha \ln B}\big)^{\frac{(1-\alpha)}{\alpha}}B^{-\frac{(1-\alpha)}{\alpha \ln B}+1},\;
c_{121}= \frac{2 \pi c_{119}}{\alpha} \Gamma(2)\big(\frac{2}{\alpha}\big)^{\alpha}c_{120},\\
c_{122}=\max(1,4c_{118}(R)c_{121},c_{1X}/R) \ge 1,\;
c_{123}=(e c_{122})^{\alpha} < 1,\;
c_{128}=\frac {1}{3^{\alpha}2}c_{123} < 1,\\
c_{110} = \big(c_{122} (8/3) \Gamma(1/\alpha)/[\alpha c_{123}^{1/\alpha}(\alpha c_{128})^{1/(\alpha-1)}]\big)^{1/2},\\
c_{109}=\min(\sqrt{\epsilon\,\delta/36}, c_{128}/2, 1) \le 1,\;
c_{130} = \displaystyle \frac{3 c_{109}}{4 \delta}\Big(\frac{1}{16 }\Big)^{5},\\
c_{131}=\max\{\sqrt{2}(16)^6,\, (\sqrt{2}(16)^6 3^{(\alpha-1)}\sqrt{\epsilon_0 \delta})/c_{123},\, ((16)^6 \sqrt{\epsilon_0 \delta})/(3\sqrt{2}) \} > 1, \\
c_{135}=r^{\alpha} c_{2X}\frac{1}{2 3^{\alpha^1}},\;
c_{132}=\min(c_{135}(r),c_{137}),\\
c_{137}=\min(\frac{1}{2}\big(c_{102}\delta^{\alpha}\frac{(c_{130})^{\alpha}}{(\sqrt{2})^{\alpha}} ,\delta \frac{c_{130}}{2 \sqrt{2}}\big) , \frac{1}{2} c_{102}\delta^{\alpha}(\frac{1}{2\sqrt{2}}c_{130})^{\alpha}).
\end{eqnarray*}
Then the coefficients dependent on $c_U, c_P, c_A$ are:
\begin{eqnarray*}
c_{116}=3\max(c^2_{1,T}\max\{c_{110},c_A\}^2\max\{ 1,\, c_P\}^2, c_{114}c_U^2,c_{115} c_U^2),\\
c_{113}=\max(c_{116}, c_U^2c_{112}(1+\tau^3_0)(1+|\chi'_1|^2_{C^0}/\delta^2)e^{12\delta\tau_0}), \\
c_{134}= c_U\Big((r c_{1X}) \frac{8}{3}\Gamma\Big(\frac{1}{\alpha}\Big)\frac{1}{\alpha (r^{\alpha} c_{2X})^{1/\alpha}(\alpha c_{135})^{1/(\alpha-1)}}\Big)^{1/2}\\
c_{136} =
2  c_{101} \sqrt{c_{113}} \int_{\R} \sqrt{(s/\delta)^2+1}\,  e^{ -c_{102}s^{\alpha}/2} ds +
\\+ c_{101} \Big(\frac{2c_{113}(1+ c_{130}^2)}{\min\{1,\,c_{130}^2/2\}}\Big)^{\frac{1}{2}} \Big(2 \int_{0}^{+\infty}  e^{- y' /2} dy'
+  \int_{\R} e^{-c_{102}|x'|^{\alpha}} dx' \Big)\\
c_{129}=\max(c_{134}, c_{136}).
\end{eqnarray*}
We rename $c_{129}$ with $c_{150}$ to underline its new dependencies.
\end{proof}

\noindent We now introduce the main assumptions to prove the global stability result.\\
We recall that the support condition in Lemma \ref{lm_expo1} is not fulfilled everywhere for $u \in H^1(\Omega)$.
The idea is that at each step one applies the local stability result of Lemma \ref{lm_expo1} in a ball centred in the point $y_j$ and then one removes from $\supp(u)$ (in a smooth way) a part of the ball $B_r(y_j)$ already calculated, for example by subtracting by $b(2(y-y_j)/r)u_j$, which is supported in $B_r(y_j)$. Then $u_{j+1}$ fulfills the support condition in Lemma \ref{lm_expo1} in the ball
$B_{2R}(y_{j+1})$, also due to our Assumption A1 or A3.

\begin{paragraph}{\bf Assumption A4}
Let $\Omega_0$ and $\psi$ be as in Assumption A1.
Then consider $r$ and $R$ the uniform radii defined in Proposition \ref{prop}. \\
We define the set of points ${\cal E}=\{y_j\in \Omega_0, \, j=1,..,N\}$,
such that $\overline{\Omega_a} \subset \bigcup_{j=1}^N B_{r}(y_j) \subset \Omega_0$,
in the following way:
\\
1. Let $y_1 \in \Omega_0$ be the maximum point for $\psi$ in $\overline \Omega_a$.\\
Set $u_1=u$ and
$u_2(y) = \Big(1-b\big(\frac{2(y-y_1)}{r}\big)\Big)u_1$,
\\
2. Let $y_2 \in \Omega_0$ be the maximum for $\psi$ in $\overline \Omega_a\backslash B_{r/2}(y_1)$,\\
3. In general, let $y_j \in \Omega_0$ be the maximum for $\psi$ in $\overline \Omega_a\backslash \bigcup_{k=1}^{j-1} B_{r/2}(y_k)$, i.e.
$y_j \in \mbox{argmax}\{ \psi(y), \, y \in (\overline \Omega_a \backslash \bigcup_{k=1}^{j-1} {B_{r/2}(y_k)}) \}$.
\\
Then we define:
\begin{eqnarray}\label{uj}
u_j = \prod_{k=1..j-1}\big(1-b_k\big)u, \quad b_{k}:=b\Big(\frac{2(y-y_k)}{r}\Big).
\end{eqnarray}
Each $y_j$ lies on the surface $S_{j}=\{y;\, \psi(y)=\psi(y_j)\}$.
Notice that, since $|y_j-y_k|\geq r/2$ for $j\neq k$,
\begin{eqnarray}\label{170}
N \le c_{170}=\frac{Vol(\Omega_0)}{\omega_{n+1} (\frac{r}{4\sqrt{\max(b_1,1)}})^{n+1}},
\end{eqnarray}
where $\omega_{n+1}$ is the volume of the ball of radius one in $\R^{n+1}$, where
we consider the following bound for the coefficients $a_1 \delta^{jk} \le g^{jk}(x) \le b_1 \delta^{jk}$.\\
We finally define $l(y) \in C^{\infty}_0(\R^{n+1})$ a localizer such that $l = 1$ on ${\cal L}=\{y\in \Omega_0;\, \dist(y,\p\Omega_0)\ge \frac{R_1}{4}\}$, $0\le l \le 1$ and supp$(l) \subset \Omega_0$.
Observe that $\cup_{k=1}^{N}B_{2R}(y_k) \subseteq {\cal L}$.
\end{paragraph}
\\\\
We now can formulate a stability estimate of inverse exponential type for the low temporal frequencies of $u_j$.

\begin{theorem}\label{iter}
Under the Assumptions A1-A2-A4,
let $y_k \in {\cal E}$
and let
$b\in G^{1/\alpha}_0(\R^{n+1})$ be a Gevrey functions of class $1/\alpha$ with compact support, such that $0<\alpha<1$.
\\
Then, there exist constants $R,r$ with $R \ge 2r >0$, and $c_{159}>1$
such that
 for $\mu > c_{159}$ there are coefficients $c_{151}, c_{152}, c_{154},c_{155},c_{156}, \beta, N$ for which,\\
if
\begin{eqnarray}\label{ipo1}
\|u\|_{H^1(\Omega_1)}= 1,\quad \|Pu\|_{L^2(\Omega_1)}
< 1
,\quad \|A\big(\frac{D_0}{\beta \mu}\big)l(y)Pu\|_{L^2} \le  \exp(-\mu^{\alpha}),
\end{eqnarray}
then calling  $\mu_1=\mu$ and $\mu_j = c_{156} \mu_{j-1}^{\alpha}$ for $2 \le j\le N$, we have $\mu_j \ge 1$ and
\begin{eqnarray}\label{ipo2}
\|u_j\|_{H^1(B_{2R}(y_j))}\le c_{152},\quad \|Pu_j\|_{L^2(B_{2R}(y_j))} \le c_{153},
\end{eqnarray}
\begin{eqnarray}\label{ipo3}
\|A(D_0/\mu_j)b((y-y_j)/R)Pu_j\|_0 \le c_{154,j}\exp({-\mu_j^\alpha}),
\end{eqnarray}
and consequently
\begin{eqnarray}\label{ipo4}
\|A(D_0/\omega)b((y-y_j)/r)u_j\|_{H^1} \le c_{155,j}\exp({-c_{132}\mu_j^{\alpha^2}}),\quad \forall\, \omega \le \mu_j^\alpha/(3c_{131}).
\end{eqnarray}
The radii $r$ and $R$ are defined in Table \eqref{tab1}, while the coefficients $c_{k}$ are calculated in the proof of the Theorem.
\end{theorem}

\begin{proof}
Let  $b\in G_0^{1/\alpha}(\R^{n+1})$ be a localizer with support as in Assumption 2.
\\
Observe that according to our definitions we have
:\\
\begin{equation*}
B_r(y_j) \subset \mbox{supp}\,b((y-y_j)/r) \subseteq B_{2r}(y_j) \subset B_R(y_j) \subset \mbox{supp}\,b((y-y_j)/R)\subseteq B_{2R}(y_j)\,,
\end{equation*}
We now proceed step by step.\\
Step 1. We consider $y_1 \in \cal{E}$ defined in Assumption A4. From the hypotheses \eqref{ipo1} the following inequalities hold true for $u_1=u$:
\begin{eqnarray*}
\|u\|_{H^1(B_{2R}(y_1))} \le 1, \quad
\|Pu\|_{L^2(B_{2R}(y_1))} \le 1.
\end{eqnarray*}
From the definition of $l$ in Assumption A4
and applying Lemma \ref{lm_util}.(b) with $f = b\big(\frac{y-y_1}{R}\big)$, $h=l$, $\beta_1=\beta$, $\mu = \beta \mu$ we get
\begin{eqnarray}\label{bj-0}
\|A\big(\frac{D_0}{\mu}\big)b\Big(\frac{y-y_1}{R}\Big) l(y)Pu\|_{L^2}&\le& \|A\big(\frac{D_0}{\beta\mu}\big)l(y)Pu\|_{L^2} + \tilde c_{107} \exp({-\tilde c_{106}\beta^{\alpha} \mu^{\alpha}}) \|l(y)Pu\|_{L^2}
\nonumber
\\&\le& \exp({-\mu^{\alpha}})
+ \tilde c_{107} \exp({-\mu^{\alpha}}) \le c_{154,1}\exp({-\mu^{\alpha}})
\end{eqnarray}
with $c_{154,1}=1 + \tilde c_{107}$ and where $ \beta >2$ is a parameter chosen as:
\begin{eqnarray}\label{beta}
\beta= 2 + \big(\frac{4}{\tilde c_{117}}\big)^{1/\alpha}
\end{eqnarray}
in order to have $\tilde c_{106}\beta^{\alpha}=1$.
Indeed, applying Lemma \ref{lm_util}.(b), one gets
$\tilde c_{106}=\frac{\tilde c_{117}}{4} \big( 1-\frac{2}{\beta}\big)^{\alpha}$
where $\tilde c_{117}=1/(2 \tilde c_3)^{\alpha}$ and $\tilde c_3=c_{1,b}(R)\cdot\mbox{Vol}(\supp(b\big(\frac{y-y_1}{R}\big)))$, with $c_{1,b}(R)$ the Gevrey parameter associated with $b\big(\frac{y}{R}\big)$. For the calculation of $\tilde c_{107}$ see lemma \ref{lm_util}.
Notice that $\tilde c_{106}$ and $\tilde c_{107}$ are independent of $y_1$, since the calculation is invariant up to translations.
\\
Calling $\tilde \psi(y) = \psi(y) - \psi(y_1)$ we notice that
$u$ fulfils $\supp(u) \cap B_{2R}(y) \subset \{y; \psi(y)\le \psi(y_1)\}$
We are then allowed to apply Lemma \ref{lm_expo1}, with $y_0=y_1$, $\psi=\tilde \psi$, $c_U=1, c_P=1, c_A=c_{154,1}$ and calling $c_{155,1}=c_{150}$,
\begin{eqnarray*}
\|A\big(\frac{D_0}{\omega}\big)b\big(\frac{y-y_1}{r}\big)u\|_{H^1} \le c_{155,1} \exp({-c_{132}\mu^{\alpha^2}}),\quad
\forall \omega \le \frac{\mu^\alpha}{3 c_{131}}.
\end{eqnarray*}
Step $j >1$.\\
Here we consider $y_j \in \cal{E}$ and $u_j$ defined in \eqref{uj}
and notice that $\mbox{supp}(u_j) \subseteq \mbox{supp}(u)\backslash \cup_{k=1}^{j-1}B_{r/2}(y_k)$ and that
$u_j = u$ on $\mbox{supp}(u)\backslash  \cup_{k=1}^{j-1}B_{r}(y_k)$.\\
Calling $\tilde \psi(y) = \psi(y) - \psi(y_j)$ we notice that
by construction $u_j$ is such that $\supp(u_j) \cap B_{2R}(y_j) \subset \{y; \psi(y)\le \psi(y_j)\}$.
We then will apply Lemma \ref{lm_expo1}, with $\psi=\tilde \psi$ and $y_0=y_j$.\\
We start by calculating the first estimate in \eqref{ipo2}:
\begin{eqnarray}\label{ipo5}
\|u_j\|_{H^1(B_{2R}(y_j))} \le \|u\|_{H^1(B_{2R}(y_j))} + |\nabla\prod_{k=0..j-1} \Big(1-b\big(\frac{2(y-y_k)}{r}\big)\Big)|_{C^0} \|u\|_{L^2(B_{2R}(y_j))}\\
\le 2\big(1+ j \frac{|b'|_{C^0}}{r}\big) \|u\|_{H^1(B_{2R}(y_j))}.\nonumber
\end{eqnarray}
Since $j\le N$ we get a uniform bound for all $j$
\begin{eqnarray}\label{c152}
\|u_j\|_{H^1(B_{2R}(y_j))} \le c_{152},\quad c_{152}=2\big(1+ N \frac{|b'|_{C^0}}{r}\big).
\end{eqnarray}
Then we consider the second estimate in \eqref{ipo2}
\begin{eqnarray}\label{ipo6}
\|Pu_j\|_{L^2(B_{2R}(y_j))} \le \|Pu\|_{L^2(B_{2R}(y_j))} + \|\big[P,\prod_{k=0..j-1} \Big(1-b\big(\frac{2(y-y_k)}{r}\big)\Big)\big] u\|_{L^2(B_{2R}(y_j))}
\\ \nonumber
\le 1+2j\big(1+ n^2|g^{kr}|_{C^0} + |h^s|_{C^0}\big)\big(\frac{|b'|_{C^0}}{r} + \frac{|b''|_{C^0}}{r^2}+ (j-1)\frac{|b'|^2_{C^0}}{r^2}\big)\|u\|_{H^1(B_{2R}(y_j))} \le c_{153},
\end{eqnarray}
where the commutator is, for $b_k=b(2(y-y_{k})/r)$:
\begin{eqnarray*}
[b_k,P]u = (-P_2 b_k)u + 2 D_0 b_k D_0 u - 2 g^{hr}(x)D_{x_h} b_k D_{x_r} u + i h^s(x) D_{x_s} b_k u
\end{eqnarray*}
with $P_2 = -D_0^2 + g^{hr}(x)D_hD_r$  and, for all $j\le N$,
\begin{eqnarray}\label{c153}
c_{153} = 1+2N\big(1+ n^2|g^{kr}|_{C^0} + |h^s|_{C^0}\big)\big(\frac{|b'|_{C^0}}{r} + \frac{|b''|_{C^0}}{r^2}+ (N-1)\frac{|b'|^2_{C^0}}{r^2}\big).
\end{eqnarray}
The third estimate \eqref{ipo3} requires information of Step $j-1$. \\
Like in \eqref{bj-0}, from the definition of $l$
and applying Lemma \ref{lm_util}.(b) with $f = b\big(\frac{y-y_j}{R}\big)$, $h=l$, $\beta_1=\beta$, $\mu = \beta \mu_{j}$ we get
\begin{eqnarray}\label{bj}
\|A\big(\frac{D_0}{\mu_j}\big)b\Big(\frac{y-y_j}{R}\Big) l(y)Pu_{j}\|_{L^2}\le \|A\big(\frac{D_0}{\beta\mu_{j}}\big)l(y)Pu_{j}\|_{L^2} + c_{153}\tilde c_{107} \exp({- \mu_{j}^{\alpha}}).
\end{eqnarray}
where $ \beta >2$ is the parameter \eqref{beta}.
\\
The first term on the right hand side of \eqref{bj} becomes
\begin{eqnarray}\label{lj}
\|A\big(\frac{D_0}{\beta\mu_j}\big)l(y)Pu_j\|_{0} \le
\|A\big(\frac{D_0}{\beta \mu_{j}}\big)l(y)Pu_{j-1}\|_{0}+\|A\big(\frac{D_0}{\beta \mu_{j}}\big)l(y)b_{j-1}Pu_{j-1}\|_{0}
\nonumber \\
+\|A\big(\frac{D_0}{\beta \mu_j}\big)[b_{j-1},P]u_{j-1}\|_{0}.
\end{eqnarray}
One can find recursively the estimate above  for $j=1$ by using \eqref{ipo1}
with $c_{162,1}=1$, and stating for $j-1$
\begin{eqnarray}\label{lj-1}
\|A\big(\frac{D_0}{\beta \mu_{j-1}}\big)l(y)Pu_{j-1}\|_{0} \le c_{162,j-1} e^{-\mu_{j-1}^\alpha},
\end{eqnarray}
with $c_{162,j-1}$ a positive parameter.\\
By the inductive hypothesis and in analogy with \eqref{bj-0},
\begin{eqnarray}\label{bj-1}
\|A\big(\frac{D_0}{\mu_{j-1}}\big)b\Big(\frac{y-y_{j-1}}{R}\Big) l(y)Pu_{j-1}\|_{0}
\le c_{154,j-1}e^{-\mu_{j-1}^\alpha},
\end{eqnarray}
where $c_{154,j-1}=c_{162,j-1} + c_{153}\tilde c_{107}$.\\
The first term on the right hand side of \eqref{lj} becomes, for
$\mu_j \le \mu_{j-1}/2$,
$$\|A\big(\frac{D_0}{\beta \mu_j}\big) l(y)Pu_{j-1}\|_0 \le \|A\big(\frac{ D_0}{\beta \mu_{j-1}}\big)l(y)Pu_{j-1}\|_0.$$
The second term on the right hand side of \eqref{lj} becomes, for $2\mu_j \le \mu_{j-1}/3$,
\begin{eqnarray*}
\|A\big(\frac{D_0}{\beta \mu_{j}}\big)l(y)b_{j-1}Pu_{j-1}\|_{0} \le
\|A\big(\frac{D_0}{\beta \mu_{j-1}}\big)l(y)Pu_{j-1}\|_{0} + c_{153} c_{164}\exp(- c_{165}  \mu_{j-1}^\alpha )
\end{eqnarray*}
where by lemma \ref{lm_util}.b)  with $f=b_{j-1}$, $h=l$, $\beta_1=3$, $\mu=\beta \mu_{j-1}$, we have $c_{164} = c_{107}$, $c_{165} = c_{106}\beta^\alpha=c_{117}\beta^\alpha/(3^\alpha 4)$. Notice that $c_{165}$ and $c_{164}$ are independent of $y_j$, since the calculation is invariant up to translations.
\\
The term with the commutator in \eqref{lj} can be split in the following way:
\begin{eqnarray*}
I_1 + I_2 + I_3 &=& \|A\big(\frac{D_0}{\beta\mu_j}\big)\Big((-P_2 b_{j-1}) + ih^s(x) D_{x_s} b_{j-1} \Big)u_{j-1}\|_{L^2}\\
&&+\|A\big(\frac{D_0}{\beta \mu_j}\big)\Big(2 D_0 b_{j-1} D_0 u_{j-1}  \Big)\|_{L^2}\\
&&+\|A\big(\frac{D_0}{\beta \mu_j}\big)\Big(2 g^{kr}(x)D_{x_k} b_{j-1} D_{x_r} u_{j-1} \Big)\|_{L^2}.
\end{eqnarray*}
We notice that the localizer $\displaystyle b\Big(\frac{(y-y_{j-1})}{r}\Big)=1$ on supp$[b_{j-1},P]u$, then we multiply $u_{j-1}$ in $I_1$ with it to keep its support in $B_{2r}(y_{j-1})$ in order to use the estimates of Step $j-1$. For $\displaystyle \nu \le \frac{\mu_{j-1}^{\alpha}}{3 c_{131}}$ a positive parameter, one has
\\
\begin{eqnarray*}
I_1 \le \|A\big(\frac{D_0}{\beta \mu_j}\big)\Big((-P_2 b_{j-1}) + ih^s(x) D_{x_s} b_{j-1} \Big)A\big(\frac{D_0}{\nu}\big)b\big(\frac{(y-y_{j-1})}{r}\big)u_{j-1}\|_{L^2}\\
+\|A\big(\frac{D_0}{\beta \mu_j}\big)\Big((-P_2 b_{j-1}) + ih^s(x) D_{x_s} b_{j-1} \Big)(1-A\big(\frac{D_0}{\nu})\big)b\big(\frac{(y-y_{j-1})}{r}\big)u_{j-1}\|_{L^2}
\\
\le c_{155,j-1} |-P_2 b_{j-1} + h^s(x)D_{x_s}b_{j-1}|_{C^0} \exp(-c_{132} \mu_{j-1}^{\alpha^2})\\
+c_{107} c_{152}\big(1+ n^2|g^{kr}|_{C^0}  + |h^s|_{C^0} \big)\exp(-c_{106}\nu^{\alpha})
\end{eqnarray*}
Notice that the first estimate on the right hand side is done by using the inductive hypothesis and by applying to the term $\|A\big(\frac{D_0}{\nu}\big)b\big(\frac{(y-y_{j-1})}{r}\big)u_{j-1}\|_{L^2}$ Lemma \ref{lm_expo1}   with coefficients
$c_U=c_{152}$, $c_P=c_{153}$, $c_A = c_{154,j-1}$ defined in \eqref{bj-1} and then
calling $c_{155,j-1}$ the resulting coefficient $c_{150}=c_{150}(c_{152},c_{153},c_{154,j-1})$.\\
For the second term on the right hand side we assume that $2\beta \mu_j \le \nu/3$ in order to write, both with
$s=0$ (i.e. $L^2$) and $s=1$ (i.e $H^1$):
\begin{eqnarray}\label{ineq}
\|A\big(\frac{D_0}{\beta \mu_j}\big) v\|_s \le \|A\big(\frac{3 D_0}{\nu}\big)v\|_s.
\end{eqnarray}
Then we apply Lemma \ref{lm_util}.(a) with $\beta_1=3$, $\mu = \nu$ and $f$ of this form (after moving out of the norm $g^{kh}, h^s$ and the complex variable)
\begin{eqnarray}\label{f1}
f_1(t,x) = \partial^2_t b_{j-1} + \partial_{x_r}\partial_{x_h}b_{j-1}+ \partial_{x_s} b_{j-1},
\end{eqnarray}
involving just  derivatives of smooth functions in $C^\infty_0(\R^n, G_0^{1/\alpha}(\R_t))$.\\
To recover an expressions for the coefficients
we recall that
the $\kappa_2-$derivative of $h\in G^{1/\alpha}_0$ (with Gevrey constant $c_{1,h}$) is:
\begin{eqnarray}\label{def_der}
|D^{\kappa_1} (D^{\kappa_2} h(s))| &\le& c_{1,h}^{|\kappa_1|+|\kappa_2|+1} (|\kappa_1|+|\kappa_2|+1)^{(|\kappa_1|+|\kappa_2|)/\alpha} \\ \nonumber
&\le& c_{1,h}^{|\kappa_2|+|\kappa_1|+1} 
2^{|\kappa_2|(|\kappa_1|+|\kappa_2|)/\alpha}
e^{|\kappa_2| |\kappa_1|/\alpha}(|\kappa_1|+1)^{|\kappa_1|/\alpha}, \quad s \in \supp(h).
\end{eqnarray}
In our case we must just consider time derivatives, in order to estimate \eqref{fourier}.
Since the translations play no role for the Fourier transform, we can calculate coefficients independently upon $j$.
Call
$c_{1,b},c_{1,b'},c_{1,b''}$
the Gevrey coefficients of the functions
$b(y),\,D_xb(y),\,D^2_xb(y)$.
Then, define
\begin{eqnarray*}
c_{f_1} =
2^{2/\alpha}e^{2/\alpha}c^2_{1,b}(r)+ c_{1,b''}(r) + c_{1,b'}(r).
\end{eqnarray*}
Analogously we can get the values for the functions associated with $I_1, \,I_2$ (see below for definition of $f_2, \,f_3$):
\begin{eqnarray*}
c_{f_2} =
2^{1/\alpha+1}e^{1/\alpha}c_{1,b}(r)+ 2^{2/\alpha+1}e^{2/\alpha}c^2_{1,b}(r),
\, \quad
c_{f_3} =
4 c_{1,b'}(r)+2 c_{1,b''}(r).
\end{eqnarray*}
In analogy to the computations above we can calculate
 $c_{Df_2}, c_{Df_3}$ (the Gevrey parameters of $Df_2= \partial_t(2 \partial_t b_{j-1})+\nabla_x (2 \partial_t b_{j-1})$ , $Df_3=
 \partial_t(2 \partial_{x_k} b_{j-1}))+\nabla_x (2 \partial_{x_k} b_{j-1})$), in order to apply Lemma \ref{lm_util}.(c) with $H^1-$norms.
\\
Now call $c_{comm} = c_{f_1} + c_{f_2} + c_{f_3} + c_{Df_2} + c_{Df_3}$ the biggest Gevrey parameter, common to all the functions $f_1,f_2,f_3$ inside the commutator, set $c_3 = c_{comm}\cdot \max_i \mbox{Vol}(\supp(f_i))$,
 and $c_{117}=1/(ec_3)^{\alpha}$.
Then, define the following coefficients in Lemma \ref{lm_util}, that are independent of the center point $y_j$:
\begin{eqnarray}\label{c106}
c_{106}=\frac{1}{(3^{\alpha}4)(ec_3)^{\alpha}},\quad
c_{107}=c_{108}=\big(c_3\frac{8}{3}\Gamma\Big(\frac{1}{\alpha}\Big)\frac{1}{\alpha (c_{117})^{1/\alpha}}\frac{1}{(\alpha c_{106})^{\frac{1}{\alpha-1}}} \big)^{1/2}.
\end{eqnarray}
Next we estimate $I_2$ moving the derivative $D_0$ of $u_{j-1}$ in front of the integrand, then multiplying $u_{j-1}$ with $b\big(\frac{y-y_{j-1}}{r}\big)$, and finally adding and subtracting operators $A(D_0/\nu)$ with $\nu\le \frac{\mu_{j-1}^{\alpha}}{3 c_{132}}$,
\begin{eqnarray*}
I_2 &\le& \|A\big(\frac{D_0}{\beta\mu_j}\big)2 D_0 b_{j-1} \Big[A\big(\frac{D_0}{\nu}\big)+ (1-A\big(\frac{D_0}{\nu}))\Big]  b\big(\frac{(y-y_{j-1})}{r}\big) u_{j-1}\|_{H^1}
\\
&&+\|A\big(\frac{D_0}{\beta\mu_j}\big)D_0(2 D_0 b_{j-1})\Big[A\big(\frac{D_0}{\nu}\big)+ (1-A\big(\frac{D_0}{\nu})\big)\Big] b\big(\frac{(y-y_{j-1})}{r}\big) \Big]u_{j-1}\|_{L^2}
\\
&\le&
 c_{155,{j-1}} |2  D_0 b_{j-1}|_{C^1}\exp(-c_{132} \mu_{j-1}^{\alpha^2})
 + c_{152} c_{108}  \exp(-c_{106}\nu^{\alpha})\\
&&+ c_{155,j-1}  |D_0(2  D_0 b_{j-1})|_{C^0} \exp(-c_{132} \mu_{j-1}^{\alpha^2})
+ c_{152}c_{107}  \exp(-c_{106}\nu^{\alpha})
%
\end{eqnarray*}
To get the estimate above we apply twice Lemma \ref{lm_expo1} with the same parameters as in $I_1$. Next using \eqref{ineq} we estimate the terms $\|A\big(\frac{3 D_0}{\nu}\big) f (1-A\big(\frac{D_0}{\nu})) v \|_s$ using Lemma \ref{lm_util} c) and a).\\
Proceeding like with $I_1$, we have to calculate the time-Fourier transform of:
\begin{eqnarray*}
f_2(y)= 2 \partial_t b_{j-1} + 2 \partial_t^2 b_{j-1},
\end{eqnarray*}
and the associated coefficients are \eqref{c106}.
Finally, moving the derivative $D_{x_r}$ of $u_{j-1}$ in front of the integrand,  multiplying $u_{j-1}$ with $b\big(\frac{y-y_{j-1}}{r}\big)$,  adding and subtracting operators $A(D_0/\nu)$ with $\nu\le \frac{\mu_{j-1}^{\alpha}}{3 c_{132}}$,
\begin{eqnarray*}
I_3&\le& \|A\big(\frac{D_0}{\beta\mu_j}\big) 2 g^{kr}(x) D_{x_k} b_{j-1} \Big[A\big(\frac{D_0}{\nu}\big)+ (1-A\big(\frac{D_0}{\nu}))\Big]  b\big(\frac{(y-y_{j-1})}{r}\big) u_{j-1}\|_{H^1}
\\&&+\|A\big(\frac{D_0}{\beta\mu_j}\big)D_{x_r}(2 g^{kr}(x)  D_{x_k} b_{j-1})\Big[A\big(\frac{D_0}{\nu}\big)+ (1-A\big(\frac{D_0}{\nu})\big)\Big]  b\big(\frac{(y-y_{j-1})}{r}\big) \Big] u_{j-1}\|_{L^2} \\
&\le&
  c_{155,j-1} |2 n g^{kr}  D_k b_{j-1}|_{C^1} \exp(-c_{132} \mu_{j-1}^{\alpha^2})
+c_{152} c_{108} n^2|g^{kr}|_{C^1}\exp(-c_{106}\nu^{\alpha})\\
&&+ c_{155,j-1} |D_r(2 g^{kr} D_k b_{j-1})|_{C^0} \exp(-c_{132} \mu_{j-1}^{\alpha^2})
+c_{107} c_{152}n^2 |g^{kr}|_{C^1}\exp(-c_{106}\nu^{\alpha}).
\end{eqnarray*}
Proceeding like with $I_1$ we have to calculate the time-Fourier transform of
\begin{eqnarray*}
f_3(y)= 2 \partial_{x_k} b_{j-1} + 2\partial_{x_k} b_{j-1} + 2  \partial^2_{x_k} b_{j-1}.
\end{eqnarray*}
and the associated coefficients are \eqref{c106}.
%
%
By collecting all the terms of the estimate for \eqref{lj}, the bound for \eqref{bj} becomes
\begin{eqnarray}\label{mixed}
\|A\big(\frac{D_0}{\mu_j}\big)b\big(\frac{y-y_j}{R}\big)Pu_{j}\|_{0}\le
\|A\big(\frac{D_0}{\beta \mu_j}\big)l(y) Pu_{j} \|_{L^2} + c_{153} \tilde c_{107} \exp(-\mu_j^\alpha) \nonumber
\\
\le c_{162,j}  \Big(\max\big(\exp(-\mu_{j-1}^{\alpha}),\exp(-c_{165} \mu_{j-1}^{\alpha}),\exp(-c_{132} \mu_{j-1}^{\alpha^2}),\exp(-c_{106}\nu^{\alpha})\big)\Big) \nonumber
\\+ c_{153} \tilde c_{107} \exp(-\mu_j^\alpha)\qquad\qquad\qquad\qquad\qquad\qquad\qquad\qquad\qquad\qquad\qquad
\end{eqnarray}
where, for all $j\ge 2,$
\begin {eqnarray*}
c_{162,j} &=& 2 c_{162,j-1}+ c_{153}c_{164}+c_{155,j-1} |-P_2 b_{j-1} + h^s(x)D_{x_s}b_{j-1}|_{C^0} \\
&&+c_{107} c_{152}\big(1+ n^2|g^{kr}|_{C^0}  + |h^s|_{C^0} \big)
\\
&& + c_{155,{j-1}} |2  D_0 b_{j-1}|_{C^1}
 + c_{152} c_{108}  \\
&&+ c_{155,j-1}  |D_0(2  D_0 b_{j-1})|_{C^0}
+ c_{152}c_{107}
\\
&&+  c_{155,j-1} |2n g^{kr}  D_k b_{j-1}|_{C^1}
+c_{152} c_{108} n^2|g^{kr}|_{C^1} \\
&&+ c_{155,j-1} |D_r(2  g^{kr} D_k b_{j-1})|_{C^0}
+c_{107} c_{152}n^2|g^{kr}|_{C^1}.
\end{eqnarray*}
In order to write \eqref{mixed} in the form
\begin{eqnarray*}
\|A\big(\frac{D_0}{\mu_j}\big)b\big(\frac{y-y_j}{R}\big)Pu_{j}\|_{0} \le c_{154,j}e^{-\mu_j^\alpha},
\end{eqnarray*}
we set in \eqref{ipo3}
\begin{eqnarray*}
c_{154,j} &=& c_{162,j} + c_{153} \tilde c_{107}
\end{eqnarray*}
and we look for $\mu_j$ of the form
$\mu_j =  c_{156,j}\mu_{j-1}^{\alpha}$ such that, for $\beta$ as in \eqref{beta}, and collecting all the constraints on $\mu_j$ used in the proof,
\begin{eqnarray}\label{156}
\mu_j \le \frac{1}{6\beta} \nu = \frac{1}{6\beta} \frac{\mu_{j-1}^{\alpha}}{3 c_{131}}, \quad \mbox{and consider} \quad \nonumber\\
c_{156}=c_{156,j} =\min \Big( \frac{1}{18\beta c_{131}},c_{132}^{1/\alpha},c_{165}^{1/\alpha}, \frac{c_{106}^{1/\alpha}}{3c_{131}}\Big).
\end{eqnarray}
The right hand side of \eqref{156} is independent upon $j$ due to the  definition of $c_{165}$, $c_{106}$ in \eqref{c106}, and the fact that $c_{132}$ and $c_{131}$
do not change during the iteration (see proof of Lemma \ref{lm_expo1}).
Therefore we define a parameter $c_{156}( < 1)$ independent of $j$ by the formula \eqref{156}.
We then estimate $\mu_j$ from below
\begin{eqnarray*}
\mu_j = c_{156} \mu_{j-1}^{\alpha} = c_{156} (c_{156} \mu_{j-2}^{\alpha})^{\alpha} =c_{156} (c_{156} (c_{156} \mu_{j-3}^{\alpha})^{\alpha})^{\alpha} \\
= c_{156}^{1+\alpha+ \alpha^2 +...+\alpha^{j-2}} \mu^{\alpha^{j-1}}\ge c_{156}^{1/(1-\alpha)} \mu^{\alpha^{j-1}}
\end{eqnarray*}
where we apply $1< \sum_{m=0}^{j-2} \alpha^m \le 1/(1-\alpha)$. Finally, in order to obtain the requested condition $\mu_j \ge 1$, we set
\begin{eqnarray*}
c_{156}^{1/(1-\alpha)} \mu^{\alpha^{j-1}} \ge 1, \; \forall j\in [1,N],
\mbox{  which implies } \mu \ge  c_{156}^{-\frac{1}{\alpha^{N-1}(1-\alpha)}}
\end{eqnarray*}
and we finally find $c_{159} = c_{156}^{-\frac{1}{\alpha^{N-1}(1-\alpha)}}$ in the line before \eqref{ipo1}.\\
By applying Lemma \ref{lm_expo1} with $c_U=c_{152}$, $c_P=c_{153}$, $c_A=c_{154,j}$, \\
$c_{155,j}=c_{150}(c_{152},c_{153},c_{154,j})$, one obtains the last inequality \eqref{ipo4} of the Theorem.
\end{proof}

\begin{remark}\label{remarks}
1. In order to work with the pseudodifferential operators $e^{-\epsilon |D_0|^2/2\tau}$ and $A(\beta |D_0|/\omega)$ one needs smooth functions in the time variable. Hence one should first operate a proper regularization in the time variable. We proceed in the same way as done in
\cite{EINT} or \cite{KKL}. Observe that the functions $u, Pu$ and $u_j, Pu_j$ are always multiplied by a smooth localizer when $A(D_0)$ and $e^{-D_0}$ are applied to them.
\\\\
2. About the construction.\\
a) Assumption A1 (and analogously A3) implies that:
$(\supp(u)\cap \Omega_0) \subset \{y; \psi(y) \le  \psi_{max}\}$; and that the level sets $\{y\in  \Omega_0; \psi(y) = c\}$, with $c  \in  [\psi_{min},\psi_{max}]$, are contained in $\overline{\Upsilon \cup \Omega_a}$.
 An example of this construction is in section \ref{sec3}.
\\
b) Assumption A1 and A3 can be relaxed in this way.
Instead of defining $\psi_{min}, \Omega_a$ (or $\Lambda_k$),
we just observe that the assumptions on $\psi, \Omega_0$ together with  $(\supp(u)\cap \Omega_0) \subset \{y; \psi(y) \le  \psi_{max}\}$ imply the existence of a non empty set $\Omega_a \subset \Omega_0$ for which Theorem \ref{global} holds.
$\Omega_a$ can be defined as $\cup_{j} B_{r}(y_j)$, with $y_j \in {\cal E}$ (see Assumption 4) such that the support condition $\supp(u_j)\cap B_{2R}(y_j) \subset \{y \in \Omega_0; \psi \le \psi(y_j)\}$ is fulfilled for every $j$. This construction requires to follow step by step the local iteration and sometimes this is difficult. That is why the a-priori knowledge that the level set $\{y\in \Omega_0\setminus \overline{\Upsilon} ;\psi(y)=\psi_{min}\}$ is strictly contained in $\Omega_0$ is useful, even if it excludes for example the case where the level sets of $\psi$ are parallel hyperplanes and $\supp(u)$ is on one side of one of them.
\\\\
3. In Theorem \ref{iter} we have worked under the assumptions
$$\|u\|_{H^1(\Omega_1)}= 1,\quad \|Pu\|_{L^2(\Omega_1)}  < 1, \quad \|A(D_0/(\beta\mu))l(y)Pu\|_{L^2} \le \exp(-\mu^{\alpha}),$$
 in order to apply Theorem \ref{global} easily. One can generalize the assumptions by setting
$$\|u\|_{H^1(\Omega_1)} \le c_{U,g},\quad \|Pu\|_{L^2(\Omega_1)}  \le c_{P,g},\quad
\|A(D_0/(\beta\mu))l(y)Pu\|_{L^2} \le c_{A,g}\exp(-\mu^{\alpha}),$$
and by changing the coefficients $c_{152},c_{153},c_{154,j},c_{155,j}$ accordingly. \\
This gives a statement of global stability of the unique continuation for low temporal frequencies.
\\\\
4.
Notice that Lemma \ref{lm_expo}, Lemma \ref{lm_expo1},  and Theorems \ref{uj}, \ref{global},\ref{global2}
can be reformulated for localizers supported on cylinders (instead of on balls), defined on cylinders
$$C_s(y_0)=\{(t,x):\, |t-t_0|\le s,\; |x-x_0|\le s\},$$
by observing that:
\begin{eqnarray*}
C_{r/\sqrt{2}}(y_0) \subset B_r(y_0) \subset B_R(y_0) \subset C_{\sqrt{2}R}(y_0).
\end{eqnarray*}
The advantage is to be able to reduce the assumptions on the regularity of the $x-$localizers.
Namely, one can replace $b(y)$ in $G_0^{1/\alpha}(\R^{n+1})$ with the product $b_{ti}(t)b_{sp}(x)$, where
 $b_{ti} \in G_0^{1/\alpha}(\R_t)$ and $b_{sp} \in C^{2}_0(\R^n)$ are  supported in $B_{2}$, equal to one in $B_1$ and $0\le b_{ti},b_{sp} \le 1$.
 \\
One can also replace the global localizer $l$ in the proof of Theorems \ref{uj} with $l(y)=\sum_{m=1}^M l_{ti,m}(t)l_{sp,m}(x)$,
with $l = 1$ in $\{y\in \Omega_0;\, \dist(y,\p\Omega_0)\ge \frac{R_1}{4}\}$ , which contains $\cup_{k=1}^N C_{\sqrt{2}R}(y_k)$, and
supp$(l) \subset \Omega_0$, and $l_{ti,m} \in C_0^{\infty}(\R)$ and $l_{sp,m} \in C_0^{2}(\R^n)$.
\end{remark}

\subsection{Proof of Theorem \ref{global}}
\begin{proof}
We consider two cases:\\
Case A. Assume $\|Pu\|_{L^2(\Omega_1)} \ge \exp({-c_{159}})\|u\|_{H^1(\Omega_1)}$, where
 $c_{159}:= {c_{156}^{-\frac{1}{(1-\alpha)}\frac{1}{\alpha^{N-1}}}} > 1$ has been defined in \eqref{156}.
 Then the estimate is trivial
\begin{eqnarray*}
\|u\|_{L^{2}(\Omega_a)} \le \|u\|_{H^1(\Omega_1)} \le \big(\ln (1+ \exp({c_{159}}))\big)^{\theta} \frac{\|u\|_{H^1(\Omega_1)}}{\ln \Big(1 + \frac{\|u\|_{H^1(\Omega_1)}}{\|Pu\|_{L^2(\Omega_1)}}\Big)^{\theta}}\,.
\end{eqnarray*}
Case B.\;
Assume now $\|Pu\|_{L^2(\Omega_1)} < \exp({-c_{159}})\|u\|_{H^1(\Omega_1)}$ and without restriction of generality take $\|u\|_{H^1(\Omega_1)}=1$. Our aim is to consider separately estimates for low and high temporal frequencies.
Let $A(D_0)$ be a pseudo-differential operator with symbol $a(\xi_0)\in G_0^{1/\alpha}$,
defined in Assumption A2.
Let  $b\in G_0^{1/\alpha}(\R^{n+1})$ be another localizer with support like in Assumption A2. \\
The parameter $\alpha \in (0,1)$ is then common to all the localizers in time, space and temporal frequency.
Let $y_j\in {\cal E}$ be the set of the points defined in Assumption A4 and consider the balls $B_r(y_j)$ centred in those points.
Observe that according to our definitions we have
:\\
\begin{equation*}
B_r(y_j) \subset \mbox{supp}\,b((y-y_j)/r) \subseteq B_{2r}(y_j) \subset B_R(y_j) \subset \mbox{supp}\,b((y-y_j)/R)\subseteq B_{2R}(y_j)\,,
\end{equation*}
Recall that $b\big(\frac{y-y_j}{r}\big) = 1$ in $B_r(y_j)$ and observe that
$u_j = u$ in $B_r(y_j) \backslash  {\cup}_{s=1}^{j-1} B_r(y_s)$, with $u_j$ defined in \eqref{uj}.
\\
We then cover $\Omega_a$ by the disjoint sets $B_r(y_j) \backslash  {\cup}_{s=1}^{j-1} B_r(y_s)$
and operate the initial estimate as follows:
\begin{eqnarray}\label{formulaA1}
\|u\|_{L^2(\Omega_a)} \le \|u\|_{L^2(B_r(y_1))}+\|u\|_{L^2(B_r(y_2) \backslash  B_r(y_1))}+\|u\|_{L^2(B_r(y_3)\backslash {\cup}_{s=1}^2 B_r(y_s))}
\nonumber
\\
+...+\|u\|_{L^2(B_r(y_N) \backslash {\cup}_{s=1}^{N-1} B_r(y_s))}
\\ \nonumber
=\|b\big(\frac{y-y_1}{r}\big)u\|_{L^2(B_r(y_1))}+
\|b\big(\frac{y-y_2}{r}\big)u_2\|_{L^2(B_r(y_2) \backslash B_r(y_1))}
\\ \nonumber
+
...+\|b\big(\frac{y-y_N}{r}\big)u_N\|_{L^2(B_r(y_N) \backslash {\cup}_{s=1}^{N-1} B_r(y_s))}
\\ \nonumber
\le \sum_{j=1}^N \|A\big(\frac{D_0}{\omega}\big)b\big(\frac{y-y_j}{r}\big)u_j\|_{L^2}
+  \sum_{j=1}^N \|\Big(1-A\big(\frac{D_0}{\omega}\big)\Big)b\big(\frac{y-y_j}{r}\big)u_j\|_{L^2}:= H_1 + H_2.
\end{eqnarray}
In the last estimate we have chosen $\omega >0$ and split all the terms in their low and high temporal component,
i.e.
$$\|b\big(\frac{y-y_j}{r}\big)u_j\|_{L^2(B_r(y_j))} \le \|A\big(\frac{D_0}{\omega}\big)b\big(\frac{y-y_j}{r}\big)u_j\|_{L^2}
+ \|\Big(1-A\big(\frac{D_0}{\omega}\big)\Big)b\big(\frac{y-y_j}{r}\big)u_j\|_{L^2}.$$
%
To estimate $H_2$ we have
\begin{eqnarray}\label{formulaA2}
\|\Big(1-A\big(\frac{D_0}{\omega}\big)\Big)b\big(\frac{y-y_j}{r}\big)u_j\|_{L^2}
\le \|(1-a(\xi_0/\omega))\mathcal{F}_{t \to \xi_0} \big(b((y-y_j)/r) u_j(y)\big) \|^2_{L^2}\quad
\\ \nonumber
\le \frac{1}{{\omega}^2}\int_{|\xi_0|>\omega}\int_{\R^{n}} |\xi_0\mathcal{F}_{t \to \xi_0} (b((y-y_j)/r)u_j(t,x)) |^2 dx d\xi_0
\le \frac{1}{{\omega}^2}\|b((y-y_j)/r)u_j(y)\|^2_{H^1}.
\end{eqnarray}
To estimate $H_1$ we first consider $\mu > c_{159}$ and we set $\|Pu\|_{L^2(\Omega_1)}= e^{-\mu}$,
that implies $\|A\big(\frac{D_0}{\zeta \mu}\big)l(y)Pu\|_{0}\le e^{-\mu^{\alpha}}$, for all $\zeta >0$. Then we choose
$\omega = \mu_N^\alpha/(3 c_{131})$ and apply Theorem \ref{iter} to each term of the sum:
\begin{eqnarray}\label{formulaA3}
\|A\big(\frac{D_0}{\omega}\big)b\big(\frac{y-y_j}{r}\big)u_j\|_{L^2} \le c_{155,N} \exp(- c_{132}\mu_N^{\alpha^2}).
\end{eqnarray}
This is possible since $\mu_N>1$ is the smallest time frequency of the set $\mu_j$, while $c_{155,N}$ is the biggest coefficient $c_{155,j}$,  j=1,..,N.\\
Collecting the two bounds and reminding that $\mu_N \ge c_{156}^{1/(1-\alpha)}\mu^{\alpha^{N-1}} > 1$, we have:
\begin{eqnarray*}
\|u\|_{L^2(\Omega_a)} &\le& N c_{155,N} \exp(- c_{132}\mu_N^{\alpha^2})
+  \frac{3 c_{131}}{\mu_N^\alpha}  \sum_{j=1}^N\|b\big(\frac{y-y_j}{r}\big)u_j\|_{H^1(\Omega_1)}\\
&\le& N c_{155,N} \exp(- c_{132}\big(c_{156}^{1/(1-\alpha)}\mu^{\alpha^{N-1}}\big)^{\alpha^2})
\\&&+  \frac{3 N c_{131}}{\big(c_{156}^{1/(1-\alpha)}\mu^{\alpha^{N-1}}\big)^\alpha}   \Big(1+ \frac{|b'|_{C^0}}{r}\Big)c_{152} \\
&\le& \frac{c_{158}}{\mu^{\alpha^{N}}}
=\frac{c_{158}}{(-\ln(\|Pu\|_{L^2(\Omega_1)}))^{\alpha^{N}}}
\le \frac{2^{\alpha^{N}} c_{158} \|u\|_{H^1(\Omega_1)}}{\Big(\ln \Big(1 + \frac{\|u\|_{H^1(\Omega_1)}}{\|Pu\|_{L^2(\Omega_1)}}\Big)\Big)^{\alpha^{N}}},
\end{eqnarray*}
where $c_{156}$ is defined in \eqref{156}
and
$$c_{158}= N c_{155,N}+ 3N c_{131} c_{152} \Big(1+ \frac{|b'|_{C^0}}{r}\Big)c_{156}^{-\alpha/(1-\alpha)}.$$
In the last step
we have applied $\ln(y) \ge \ln(1+y)/2$ for $y=\|u\|_{H^1(\Omega_1)}/ \|Pu\|_{L^2(\Omega_1)}>e$, and then we have returned to the original notation.
\\
Now we  choose $\alpha$ such that $\alpha = (\theta)^{1/N}$ and which belongs to $(0,1)$ so that,
defining $c_{160}=\big(\ln (1+ e^{c_{159}(\theta)})\big)^{\theta}+2^{\theta} c_{158}(\theta)$, we obtain the result.
\end{proof}
\\
In the previous theorem the dependency of $c_{160}$ upon $\theta$ is very bad.\\
For some applications it is better to keep $\alpha$ and $N$ independent and formulate the following consequence:
\begin{corollary}\label{global3}
Consider the assumptions of Theorem \ref{global}.
Then for every $0<\alpha < 1$ we have
\begin{eqnarray*}
 \|u\|_{L^{2}(\Omega_a)} \le c_{160} \frac{\|u\|_{H^{1}(\Omega_1)}}{\Big(\ln\Big(1+\frac{\|u\|_{H^{1}(\Omega_1)}}{\|Pu\|_{L^{2}(\Omega_1)}}\Big)\Big)^{\alpha^{N}}}
\end{eqnarray*}
with $c_{160}=\big(\ln (1+ e^{c_{159}})\big)^{\alpha^{N}}+2^{\alpha^{N}} c_{158}$. Here $N \le c_{170}$ given by \eqref{170}.
\end{corollary}

\subsection{Proof of Theorem \ref{global2}}

\begin{center}
\psfrag{1}{$\Upsilon$}
\psfrag{2}{\hspace{-1mm}$\Lambda_1$}
\psfrag{3}{\hspace{-1mm}$\Omega_{0,1}$}
\psfrag{4}{\hspace{-1.5mm}$\Lambda_2$}
\psfrag{5}{\hspace{-1mm}$\Omega_{0,2}$}
\psfrag{6}{\hspace{-2mm}$\Lambda_3$}
\psfrag{7}{\hspace{-2mm}$\Omega_{0,3}$}
\includegraphics[height=6cm]{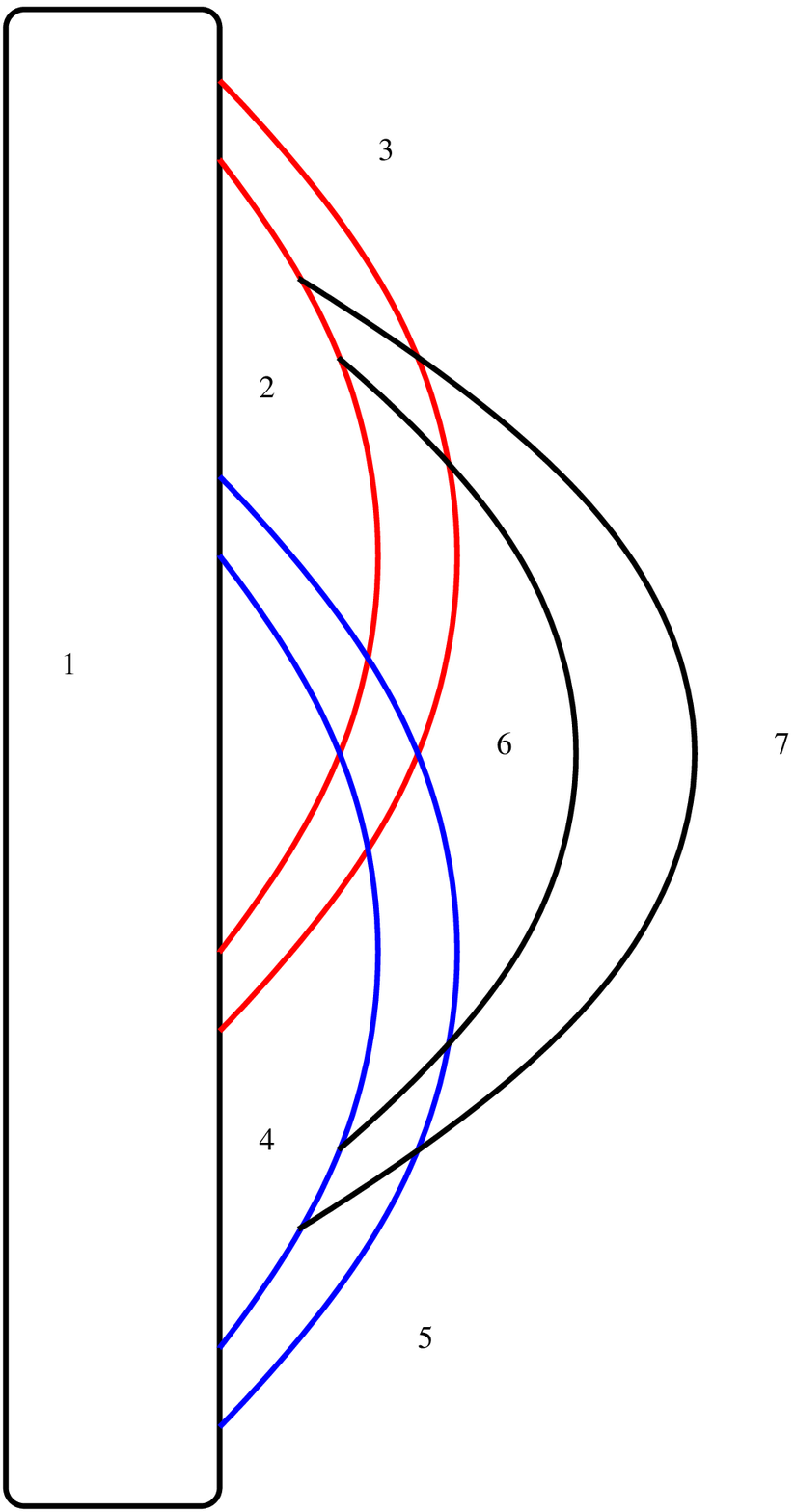}
\end{center}
{\it FIGURE 1. A possible construction of the domains $\Upsilon, \Omega_{0,k}, \Lambda_k$ }

\begin{proof}
{\it Initialization of the radii and the localizers}:
Let $A(D_0)$ be a pseudo-differential operator with symbol $a(\xi_0)\in G_0^{1/\alpha}(\R)$,
defined in Assumption A2.
We define the localizer $b(y) \in G_0^{1/\alpha}(\R^{n+1})$ with support like in Assumption A2.
\\
Using Assumption A3, in each domain $\Omega_{0,k}$ we can calculate a table like \eqref{tab1}, with $\Omega_{0,k}$ in place of $\Omega_0$, and $\Lambda_k$ in place of $\Omega_a$, and where all the constants dependency is described in Proposition \ref{prop}.
By comparing the tables of the several $\Omega_{0,k}$ we can
consider $R_2=\min_k R_{2,k}$ and find $R = \frac{1}{4}\Big(16 + \frac{1}{16}\Big)^{-1/2} R_2$ the common radius for the local stability in $\Lambda_k$. After fixing $R$, we reduce
also the values $r_k$ so that $r=\min_k r_k$ is the common radius of the ball where the $L^2$ local estimate can be performed in $\Lambda_k$.
\\
{\it Construction of the set}  {\cal E} {\it and the functions $u_j$:} For $\Omega_0=\Omega_{0,1}$, $\Omega_a=\Lambda_1$ and $\psi=\psi_1$, we define $y_j\in {\cal E}_1$ the set of the maximal points for $\psi_1$, according to the procedure in Assumption A4.
Call $N_1$ the number of points of the covering of $\Lambda_1$, i.e. $\Lambda_1 \subset \cup_{j=1}^{N_1} B_r(y_j)$.
Then we remove $\Lambda_1$ from $\Omega$ and we restart the procedure with the set $\Lambda_2$. \\
Namely,
for $\Omega_0=\Omega_{0,2}\backslash \Lambda_1$, $\psi=\psi_2$ and $\Omega_a = \Lambda_2 $, we define $y_j\in {\cal E}_2$ according to the procedure of Assumption A4, where we use the indexing $j=N_1+1,..,N_2$.\\
When also $\Lambda_2$ is covered, one skips to $\Lambda_3$ and so on.\\
At the end one can define the set of points ${\cal E}= \cup_{k=1}^K{\cal E}_k$ of cardinality $N=\sum_{k=1}^K N_k$.\\
Consider the balls $B_r(y_j)$ centred in those points and define $u_j$ as in \eqref{uj}.\\
Observe that according to our definitions we have
:
\begin{equation*}
B_r(y_j) \subset \mbox{supp}\,b((y-y_j)/r) \subseteq B_{2r}(y_j) \subset B_R(y_j) \subset \mbox{supp}\,b((y-y_j)/R)\subseteq B_{2R}(y_j)\,,
\end{equation*}
Moreover $b\big(\frac{y-y_j}{r}\big) = 1$ in $B_r(y_j)$ and observe that
$u_j = u$ in $B_r(y_j) \backslash  {\cup}_{s=1}^{j-1} B_r(y_s)$, with $u_j$ defined in \eqref{uj}.
\\
{\it Uniform parameters:}
With the previous assumptions Lemma \ref{lm_expo1} can then be applied in each ball $B_{2R}(y_j)$, with $y_j \in {\cal E}_k$ in place of $y_0$, and we call $c_{131,k}, c_{132,k}, c_{165,k}, c_{150,k}$ the related parameters, that are constant for every ball in the region $\Omega_{0,k}$, but in principle change from region to region.\\
Therefore the following uniform constants are introduced:
\begin{eqnarray}\label{156-1}
c_{131,*}= \max_{k=1..K}c_{131,k}, \;
c_{132,*}= \min_{k=1..K}c_{132,k},\;
c_{165,*}= \min_{k=1..K}c_{165,k},\;
\nonumber
\\
c_{156,*} =\min \Big( \frac{1}{18 \beta c_{131,*}},c_{132,*}^{1/\alpha}, c_{165,*}^{1/\alpha},\frac{c_{106}^{1/\alpha}}{3c_{131,*}}\Big).\;\;
\end{eqnarray}
We define $l(y) \in C^{\infty}_0(\R^{n+1})$ a localizer such that $0\le l \le 1$,  $l = 1$ on the set $\{y\in \bigcup_{k=1}^K \Omega_{0,k};\, \dist(y,\p\Omega_0)\ge \frac{R_1}{4}\}$, which contains  $\cup_{j=1}^{N}B_{2R}(y_j)$, and supp$(l) \subset \Omega_0$.
\\
In particular we consider $\beta$ as in \eqref{beta}, and the related $\tilde c_{106}$.\\
{\it Construction:}\\
We consider two cases:\\
Case A. Assume $\|Pu\|_{L^2(\Omega_1)} \ge \exp({-c_{159,*}})\|u\|_{H^1(\Omega_1)}$, where
 $c_{159,*}:= {c_{156,*}^{-\frac{1}{(1-\alpha)}\frac{1}{\alpha^{N-1}}}} > 1$.
 Then the estimate is trivial
\begin{eqnarray*}
\|u\|_{L^{2}(\Lambda)} \le \|u\|_{H^1(\Omega_1)} \le \big(\ln (1+ \exp({c_{159}}))\big)^{\theta} \frac{\|u\|_{H^1(\Omega_1)}}{\ln \Big(1 + \frac{\|u\|_{H^1(\Omega_1)}}{\|Pu\|_{L^2(\Omega_1)}}\Big)^{\theta}}\,.
\end{eqnarray*}
Case B.\;
Assume now $\|Pu\|_{L^2(\Omega_1)} < \exp({-c_{159,*}})\|u\|_{H^1(\Omega_1)}$ and without restriction of generality take $\|u\|_{H^1(\Omega_1)}=1$. Our aim is to consider separately estimates for low and high temporal frequencies.
We cover $\Lambda$ by the disjoint sets $B_r(y_j) \backslash  {\cup}_{s=1}^{j-1} B_r(y_s)$
and operate the initial estimate as follows:
\begin{eqnarray*}
\|u\|_{L^2(\Lambda)} \le \|u\|_{L^2(B_r(y_1))}+\|u\|_{L^2(B_r(y_2) \backslash  B_r(y_1))}
+...+\|u\|_{L^2(B_r(y_N) \backslash {\cup}_{s=1}^{N-1} B_r(y_s))}
\\
=\|b\big(\frac{y-y_1}{r}\big)u\|_{L^2(B_r(y_1))}
+
...+\|b\big(\frac{y-y_N}{r}\big)u_N\|_{L^2(B_r(y_N) \backslash {\cup}_{s=1}^{N-1} B_r(y_s))}
\\
\le \sum_{j=1}^N \|A\big(\frac{D_0}{\omega}\big)b\big(\frac{y-y_j}{r}\big)u_j\|_{L^2}
+  \sum_{j=1}^N \|\Big(1-A\big(\frac{D_0}{\omega}\big)\Big)b\big(\frac{y-y_j}{r}\big)u_j\|_{L^2}:= H_1 + H_2.
\end{eqnarray*}
In the last estimate we took $\omega >0$ and split all the terms in their low and high temporal component,
i.e.
$$\|b\big(\frac{y-y_j}{r}\big)u_j\|_{L^2(B_r(y_j))} \le \|A\big(\frac{D_0}{\omega}\big)b\big(\frac{y-y_j}{r}\big)u_j\|_{L^2}
+ \|\Big(1-A\big(\frac{D_0}{\omega}\big)\Big)b\big(\frac{y-y_j}{r}\big)u_j\|_{L^2}.$$
%
To estimate $H_2$ we have
\begin{eqnarray*}
\|(1-A( D_0/\omega)) b((y-y_j)/r) u_j(y) \|^2_{L^2}
\le \frac{1}{{\omega}^2}\|b((y-y_j)/r)u_j(y)\|^2_{H^1}.
\end{eqnarray*}
To estimate $H_1$ we first observe that supp$(b\big(\frac{y-y_j}{r}\big)u_j)$ is in $\Omega_{0,1}$ for $j=1,.., N_1$, where $\psi_1$ is defined. Then
supp$(b\big(\frac{y-y_j}{r}\big)u_j)$ is in $\Omega_{0,2}$ for $j=N_1+1,.., N_2$ where $\psi_2$ is defined, and so on.
\\
Consider $\mu > c_{159,*}$ and we set $\|Pu\|_{L^2(\Omega_1)}= e^{-\mu}$,
that implies $\|A\big(\frac{D_0}{\zeta \mu}\big)l(y)Pu\|_{0}\le e^{-\mu^{\alpha}}$, for all $\zeta >0$.\\
We also set $c_{152}=2\big(1+ N \frac{|b'|_{C^0}}{r}\big)$, $c_{153} = 1+2N\big(1+ n^2|g^{kr}|_{C^0} + |h^s|_{C^0}\big)\big(\frac{|b'|_{C^0}}{r} + \frac{|b''|_{C^0}}{r^2}+ (N-1)\frac{|b'|^2_{C^0}}{r^2}\big)$, and $\mu_j = c_{156,*}\mu_{j-1}^{\alpha}$. Then we choose
$\omega = \mu_N^\alpha/(3 c_{131,*})$ and apply Theorem \ref{iter} to each term of the sum:
$$\|A\big(\frac{D_0}{\omega}\big)b\big(\frac{y-y_j}{r}\big)u_j\|_{L^2} \le c_{155,N} \exp(- c_{132,*}\mu_N^{\alpha^2}).$$
This is possible since $\mu_N>1$ is the smallest time frequency of the set $\mu_j$, while $c_{155,N}$ is the biggest coefficient $c_{155,j}$,  j=1,..,N.
\\
Collecting the two bounds and reminding that $\mu_N \ge c_{156,*}^{1/(1-\alpha)}\mu^{\alpha^{N-1}} > 1$, we have:
\begin{eqnarray*}
\|u\|_{L^2(\Lambda)} &\le& N c_{155,N} \exp(- c_{132,*}\mu_N^{\alpha^2})
+  \frac{3 c_{131,*}}{\mu_N^\alpha}  \sum_{j=1}^N\|b\big(\frac{y-y_j}{r}\big)u_j\|_{H^1(\Omega_1)}\\
&\le& \frac{c_{158,*}}{\mu^{\alpha^{N}}}
=\frac{c_{158,*}}{(-\ln(\|Pu\|_{L^2(\Omega_1)}))^{\alpha^{N}}}
\le \frac{2^{\alpha^{N}} c_{158,*} \|u\|_{H^1(\Omega_1)}}{\Big(\ln \Big(1 + \frac{\|u\|_{H^1(\Omega_1)}}{\|Pu\|_{L^2(\Omega_1)}}\Big)\Big)^{\alpha^{N}}},
\end{eqnarray*}
where
$$c_{158,*}= N c_{155,N}+ 3N c_{131,*} c_{152} \Big(1+ \frac{|b'|_{C^0}}{r}\Big)c_{156,*}^{-\alpha/(1-\alpha)}.$$
\\
%
Now we  choose $\alpha$ such that $\alpha = (\theta)^{1/N}$ and which belongs to $(0,1)$ so that,
defining $c_{161}=\big(\ln (1+ e^{c_{159,*}(\theta)})\big)^{\theta}+2^{\theta} c_{158,*}(\theta)$, we obtain the result.
\end{proof}


\section{Applications}\label{sec3}

\begin{paragraph}{\bf Assumption A5}
Assume that  $M=\R^n$ and on $M$ we have a metric tensor $g$ satisfying
\beq\label{q 4b summary}
a_0 I\leq [g_{jk}(x)]_{j,k=1}^n\leq b_0\, I,\quad \hbox{and}\quad \|g_{jk}\|_{C^4(M)}\leq b_3.
\eeq
Note that then
\beq\label{q 4b summary 2}
& & b_0^{-1} I\leq [g^{jk}(x)]_{j,k=1}^n\leq a_0^{-1} I,\\ \nonumber
& & \|g^{jk}\|_{C^1(M)}\leq b_3 a_0^{-2},\\ \ \nonumber
& &\sqrt{a_0}d_{\R^n}(x_1,x_2)\leq d_{g}(x_1,x_2)\leq \sqrt{b_0} d_{\R^n}(x_1,x_2),\\ \ \nonumber
& &\frac 1{\sqrt{b_0}}\|\xi\| _{\R^n}\leq \|\xi\| _{g}\leq \frac 1{\sqrt{a_0}} \|\xi\| _{\R^n},\quad \xi\in T_x^*M,\\
\nonumber
& &{\sqrt{a_0}}\|\eta\| _{\R^n}\leq \|\eta\| _{g}\leq {\sqrt{b_0}} \|\eta\| _{\R^n},\quad \eta\in T_xM.
\eeq
Here we assume that $a_0<1<b_0$.
In Appendix A we call $a_1=b_0^{-1}$ and $b_1=a_0^{-1}$.
Note that (\ref{q 4b summary}) implies
that
\ba
|\hbox{Sec}|<\Lambda_M=\Lambda_M(a_0,b_0,b_3),
\ea
where Sec is the sectional curvature of $(M,g)$.
Also assume that the injectivity radius of $(M,g)$
satisfies $\hbox{inj}(M,g)>i_0$ with $0<i_0<\min(1,\pi/\Lambda_M^{1/2})$.
\\
Consider the wave operator \eqref{wave_op}.
Assume that the lower order coefficients are such that
$$\|h^j\|_{C^0(M)}+\|q\|_{C^0(M)} \le b_3.$$
We fix the three positive parameters $\ell,T,\gamma$ as follows:
$$\ell \le i_0/4, \quad T>\ell, \quad 0< \gamma \le T-\ell.$$
\end{paragraph}
\medskip

\noindent
In this section, we use the following definitions.
\begin{definition}\label{def3}
$B_g(z,r_1)\subset M=\R^n$ is the ball with center $z$ and radius $r_1$, defined using the Riemannian
metric $g$. Also, $B_{\R^n}(x,r_1)$ is the Euclidean ball in $\R^n$.
For $y=(t,x)\in \R\times M$, let
\beq\label {def of C}
\cC_g(y,r_1)=(t-r_1,t+r_1)\times B_g(x,r_1)
\eeq
and $\cC(y,r_1)=\cC_{\R^{n+1}}(y,r_1)=(t-r_1,t+r_1)\times B_{\R^n}(x,r_1)$.
Also, denote
\ba
d_{\R\times (\R^n,g)}((t_1,x_1),(t_2,x_2))=\max(|t_1-t_2|,d_g(x_1,x_2))
\ea
and
\ba
d_{\R\times (\R^n,e)}((t_1,x_1),(t_2,x_2))=\max(|t_1-t_2|,d_{\R^n}(x_1,x_2)).
\ea
%
Let $z\in M$, and define the {\it hyperbolic function} as
and
\beq\label{def of psi}
\psi(t,x;T,z)=\big(T-d_g(x,z)\big)^2-t^2,\quad y=(t,x)\in \R\times  \R^n.
\eeq
Note that as inj$(M)>i_0,$ $(t,x)\mapsto \psi(t,x;T,z)$ is smooth in $\R\times (B_g(z,i_0)\setminus \{z\})$.\\
Define
\beq\label{def of S}
S_{\ell,\gamma}&=&S(z,\ell,T,\gamma)\\
&:=& \nonumber
\{(t,x)\in [-T+\ell,T-\ell]\times   \R^n;\
\psi(t,x;T,z)\ge \gamma^2,\ d_g(x,z)\leq T\}.
\eeq
Also let $z\in M$ and
 \beq\label{double cone cut}
& & \Sigma(z,\ell,T) = \{(t,x)\in \R\times   \R^n;\ |t|\le T-\ell ,\ |t| \leq  T - d_g(x,z)\}
 \eeq
 be the {\it domain of influence} of the cylinder
 \beq\label{def of W}
 W(z,\ell,T)=(-T+\ell,T-\ell)\times B_g(z,\ell).
 \eeq
\end{definition}

\subsubsection
{\bf Some geometric estimates for domains of influences}

\begin{center}
\includegraphics[height=5.5cm]{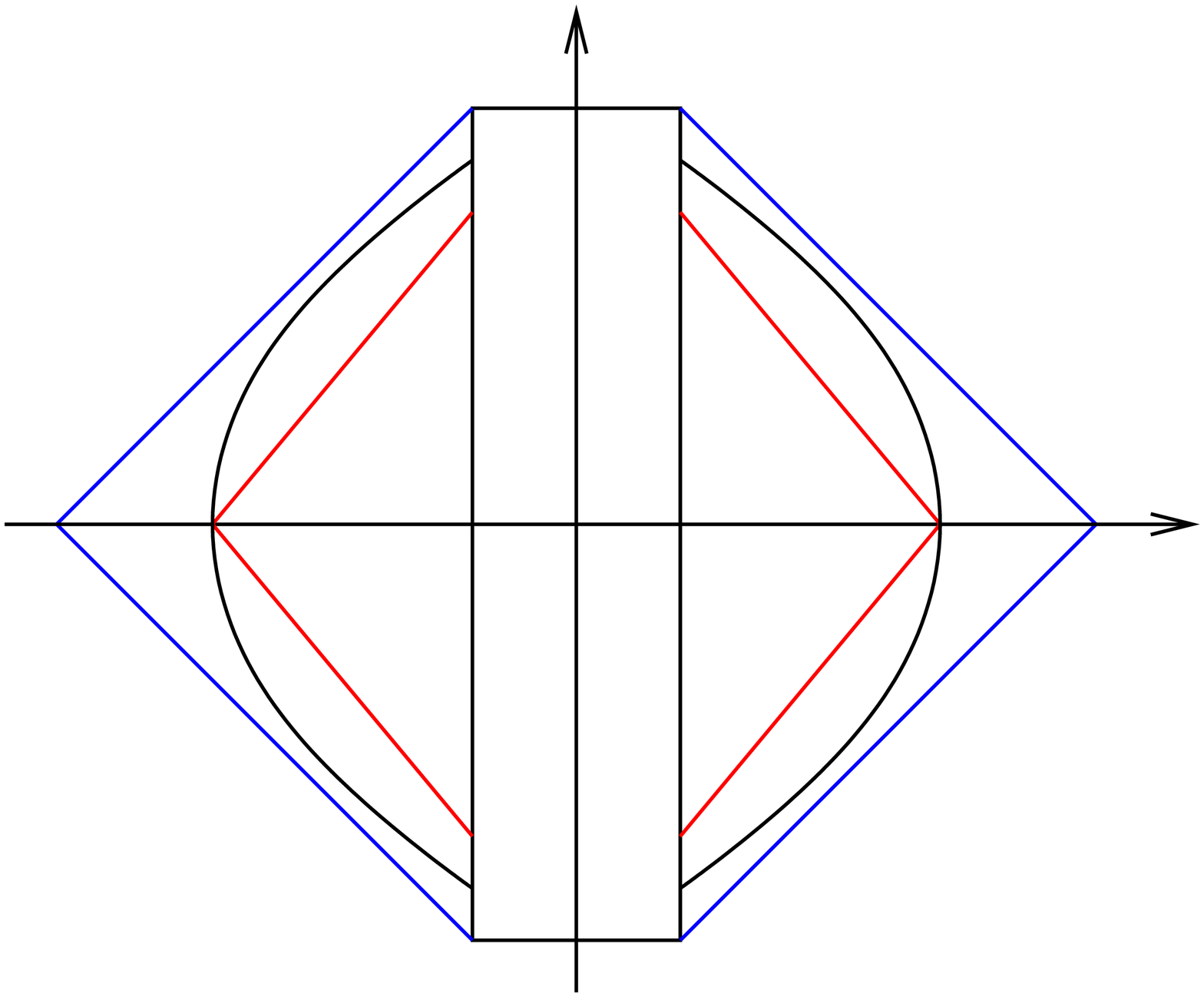}
\end{center}
{\it FIGURE 2. The hyperbolic surface between two domains of influence}

\begin{lemma}\label{lem: domains of influence}
Let $T,\ell,\gamma$ be as in Assumption A5.
Denote $T_{\ell,\gamma}=((T-\ell)^2-\gamma^2)^{\frac 12}+\ell$,
 Then
\ba
\Sigma(z,\ell,T-\gamma) \subset S(z,\ell,T,\gamma) \subset
\Sigma(z,\ell,T_{\ell,\gamma})\cup W(z,\ell,T)
\subset \Sigma(z,\ell,T).
\ea

\end{lemma}
{\bf Proof.}
Assume that $x\in \R^n$ is such that
 $s=d_g(x,z)\in [\ell,T-\gamma]$ and that
$
|t|\le T-\gamma-s.
$
Then
\ba
& &\hspace{-5mm}\psi(t,x;T,z)=\big(T-s\big)^2-t^2\\
&\ge& (T-s)^2-(T-\gamma-s)^2\\
&\geq&\bigg((T-s)-(T-\gamma-s)\bigg) \bigg((T-s)+(T-\gamma-s)\bigg) \\
&\geq&\gamma(2(T-s)-\gamma)\\
 &\geq&\gamma^2.
\ea
We see that
\ba
S(z,\ell,T,\gamma)&\supset& \{(t,x)\in [-T+\ell,T-\ell]\times \R^n;\ d_g(x,z)\ge \ell,\
|T-\gamma-d_g(x,z)|\geq |t| \}.
\ea
This proves
\beq\label{first inclusion}
\Sigma(z,\ell,T-\gamma) \subset S(z,\ell,T,\gamma).
\eeq
%
Assume next that
 $s=d_g(x,z)\in [\ell,T_{\ell,\gamma}]$ and
$
|t|> T_{\ell,\gamma}-s.
$
Then $T_{\ell,\gamma}-T<0$ implies that
\ba
\psi(t,x;T,z)&=&\big(T-s\big)^2-t^2\\
&<& (T-\ell-(s-\ell))^2-((T_{\ell,\gamma}-\ell)-(s-\ell))^2\\
&\leq&( (T-\ell)^2-(T_{\ell,\gamma}-\ell)^2)
-2(T-\ell)(s-\ell)+2(T_{\ell,\gamma}-\ell)(s-\ell)\\
&\leq& \gamma^2. 
\ea
Thus, we see that the complement of $S_{\ell,\gamma}$ satisfies
\ba
 \{(t,x)\in [-T+\ell,T-\ell]\times \R^n ;\
|T_{\ell,\gamma}-d_g(x,z)|<|t| \}\subset S_{\ell,\gamma}^c\cup W(z,\ell,T).
\ea
Hence we see that
\ba
S_{\ell,\gamma}\setminus W(z,\ell,T)\subset
\Sigma(z,\ell,T_{\ell,\gamma}) \setminus W(z,\ell,T).
\ea
This  and (\ref{first inclusion}) yield the claim.

\hfill$\square$ \medskip

\subsection{Applications:
Stability on the domain of influence of a cylinder}

Here we consider the case when the solution of the wave equation \eqref{wave_op} vanishes in the cylinder $W(z_0,\ell,T)$
and $T$ may be so large that we have to consider also singular points for $d_g$.
We refer to Definition \ref{def3} for the definition of sets used.
\\
Our aim is to prove the following:

\begin{theorem}\label{global2 application} Under the conditions of Assumption A5,
let $z_0\in \R^n$, and define
\ba
\Omega=(-T,T)\times M, \;\Upsilon = W(z_0,T,\ell),\; \Lambda=S(z_0,\ell,T,\gamma)\setminus \Upsilon,\\
\Omega_0= S(z_0,\ell,T,\frac{\gamma}{\sqrt{2}}) \setminus \{(t,x): t\in \R,  d_g(z_0,x) \le \frac{\ell}{4}  \},
\; \Omega_1= \Omega_0 \setminus \Upsilon.
\ea
Assume that $u \in H^1(\Omega)$ satisfies
\ba
P(x,D) u(y)=f(y),\quad \hbox{for}\; y\in \Omega
\ea
and
\beq\label{eq: original support}
& &u|_{W(z_0,\ell,T)}=0.
\eeq
Then for every $0<\theta < 1$ we have
\begin{eqnarray*}
 \|u\|_{L^{2}(\Lambda)} \le c_{163} \frac{\|u\|_{H^{1}(\Omega_1)}}{\Big(\ln\Big(e+\frac{\|u\|_{H^{1}(\Omega_1 )}}{\|f\|_{L^{2}(\Omega_1 )}}\Big)\Big)^{\theta}}.
\end{eqnarray*}
Here, $c_{163}$ depends only on $a_0,b_0,b_3,T,
\gamma,\ell$, $i_0,$  and $\theta$.
\end{theorem}

\begin{corollary}\label{cor2}
By Lemma \ref{lem: domains of influence} we observe that, after a reparametrization of the time,
$\Sigma(z_0,\ell,T) \subset S(z_0,\ell,T+\gamma,\gamma)$. Consider the wave equation formulated in Theorem \ref{global2 application}. Hence for each $\gamma$ such that $0<\gamma < T- \ell$, the optimal time of the control $T-\ell$ (with $T-\ell =\max_{x,y \in \Sigma(z_0,\ell,T)\backslash W(z_0,\ell,T)} d_g(y,x)$) can be approximated from above by $T-\ell+\gamma$, using a result of stability of the unique continuation.
\end{corollary}

\subsubsection{Local stability estimate}

Below, we say that the cut-off function corresponding to a center point $\hat y=(\hat t,\hat x)\in \R\times \R^n$ and a radius $\hat r$
is the product of a ``time-variable cut-off function'' and ``space-variable cut-off function'', given by
\beq\label{def of b 1}
& &b_{\hat y,\hat r}(t,x)=b^{(ti)}_{\hat t,\hat r}(t)b^{(sp)}_{\hat x,\hat r}(x),\\ \nonumber
& &b^{(ti)}_{\hat t,\hat r}(t)=f^{(ti)}(\frac {t-\hat t}{\hat r}),\quad f^{(sp)}_{\hat x,\hat r}(x)=f^{(1)}(\frac {x-\hat x}{\hat r}),
\eeq
where
$ f^{(sp)} \in C^2_0(\R^{n})$ and $f^{(ti)}\in G^{1/\alpha}_0(\R^1)$.
We assume that $0\leq f^{(sp)}\leq 1$ and $0\leq f^{(ti)}\leq 1$.
We assume that
\ba
& &\supp(f^{(ti)})\subset B_{\R}(0,\sqrt 2),\quad f^{(ti)}|_{B_{\R}(0,1)}=1,\\
& &\supp(f^{(sp)})\subset B_{\R^n}(0,\sqrt 2),\quad f^{(sp)}|_{B_{\R^n}(0,1)}=1.
\ea
Then we have for $\hat y\in \R^{n+1}$ and $\hat r>0$
\beq\label{def of b 2}
\supp(b_{\hat y,\hat r})\subset  \cC(\hat y,2\hat r), \quad b_{\hat y,\hat r}|_{\cC(\hat y,\hat r)}=1.
 \eeq
Note that by (\ref{q 4b summary}) and (\ref{def of C}),
 $\cC(\hat y,2\hat r)
\subset
\cC_g(\hat y,2\sqrt{b_0}\hat r)$ and  $\cC_g(\hat y,\sqrt {a_0}\hat r)\subset \cC(\hat y,\hat r)$.
%
\\\\
For the proof of the global stability we must define the following points $\hat y$ and functions $\psi_{\hat z, \hat T}$.
\begin{definition}\label{def-yhat} (see Figure 2 below, where $\hat y = y_j$.)
Let  $\quad \hat y \; = \; (\hat x, \hat t) \; \in \;
S(z_0,\ell,T,\gamma) \backslash \{y; \; t\in \R;  d_g(z_0,\hat x) < \ell \}$:
\\
a) If $\ell \le  d_g(z_0,\hat x) \le \frac{7}{8}i_0$, then define $\hat z = z_0$, $\hat T = T$, $\psi_{\hat z, \hat T}(y) = \psi(y;z_0, T)$,\\
b) If $\frac{7}{8}i_0 < d_g(z_0,\hat x) $, then let  $\psi_{\hat z, \hat T}(y) = \psi(y;\hat z,\hat T)$. Calling $\gamma_{z_0,\hat \xi}$  a distance minimizing, unit speed geodesic from $z_0$ to $\hat x$ in $M$, we define $\hat z,\hat T$ as follows:
\ba
\hat L := d_g(z_0,\hat x) - \frac{i_0}{4} > \frac{5i_0}{8},\qquad \hat T := T- \hat L, \qquad
\hat z := \gamma_{z_0,\hat \xi}(\hat L),
\ea
\end{definition}
Note that the choice of the point $\hat z$ is not unique as there may be several distance minimizing geodesics from $z_0$
to $\hat x$.

\noindent Observe that for $\ell < i_0/4$ and $T > \frac{7i_0}{8}$, we have
\beq \label{property x5}
\gamma + \ell < \gamma + \frac{i_0}{4} \le \hat T=T- d_g(z_0,\hat x) + \frac{i_0}{4} < T - \frac{5i_0}{8} < T - \ell - \frac{3i_0}{8}.
\eeq
We then introduce the sets:
\beq\label{notations 1}
& &
\Omega_2(\hat z,\hat T,\ell,\gamma)=
S(\hat z, \hat T,\ell,\gamma)\cap \{(t,x): t\in \R,  \ell \le d_g(\hat z,x) \le \frac{5}{8} i_0\},
\\ \nonumber
& & \Omega_3(\hat z,\hat T,\ell,\gamma)=S(\hat z,\hat T,\ell,\frac {\gamma}{ \sqrt 2}) \cap \{(t,x): t\in \R,  \frac{1}{4} \ell \le d_g(\hat z,x) \le \frac{7}{8} i_0\}.
\eeq

\begin{lemma}\label{lm_expo1 application} Under the Assumptions A2, A5, let $\hat y$ and $\psi_{\hat z, \hat T}$ be as in 
Definition \ref{def-yhat}.
Then there exist
$r,R$, $c_U,\,c_P,\,c_A$, and
$c_{150}, c_{131}, c_{132}$,
dependent only on the parameters $a_0,b_0,b_3,T,
\gamma,\ell$, and $i_0$ such that the following property holds:
\\
If $v\in H^1_0(\Omega)$ and  $\mu \ge 1$ is such that
\beq\label{support condition}
\supp(v) \cap \cC(\hat y,2R) \subset \{y\in  \cC(\hat y,2R);\ \psi_{\hat z, \hat T}(y)\leq \psi_{\hat z, \hat T}(\hat y)\}
\eeq
and
\begin{eqnarray*}
\|v\|_{H^1(\cC(\hat y,2R))}\le c_U,\quad \|Pv\|_{L^2(\cC(\hat y,2R))} < c_P, \quad \|A(D_0/\mu)(b_{\hat y,R} Pv)\|
_{L^2(\R^{n+1})} \le c_A e^{-\mu^\alpha},
\end{eqnarray*}
then,
\begin{eqnarray*}
\|A(D_0/\omega)(b_{\hat y,r}v)\|_{H^1(\R^{n+1})} \le c_{150} \exp({-c_{132}\mu^{\alpha^2}}),\quad \hbox{for all}\ \omega \le \mu^\alpha/(3c_{131}).
\end{eqnarray*}
\end{lemma}
\begin{proof}
In the Appendix we have calculated uniform estimates for the function $\psi(y;z_0,T)$ defined in
Definition \ref{def-yhat} a);
see \eqref{property 0},  \eqref{property 1} (with $\gamma_I=\gamma/\sqrt{2}$), \eqref{property 2}, \eqref{property x1}, \eqref{property x2}.
Analogously, one can estimate the functions $\psi(y;\hat z, \hat T)$ defined in
Definition \ref{def-yhat} b):
calling $\Omega_2= \Omega_2(\hat z,\hat T,\ell,\gamma) $ and $\Omega_3= \Omega_3(\hat z,\hat T,\ell,\gamma)$, we have
\beq \label{eq: estimat}
& &\|  \psi_{\hat z,\hat T}'\|_{C^0(\Omega_3)}\leq b_4(T+1) ,\nonumber
\\
& &\|  \psi_{\hat z,\hat T}''\|_{C^0(\Omega_3)}\leq b_4 (T+1)((\ell/4)^{-1}+1) ,\nonumber
\\ \nonumber
&& \|  \psi_{\hat z,\hat T}'''\|_{C^0(\Omega_3)}\leq b_4(T+1)((\ell/4)^{-2} +(\ell/4)^{-1}),
\\ \nonumber
& &\min_{y\in \Omega_3} |d\psi_{\hat z,\hat T}(y)| \geq {\sqrt 2} \gamma b_0^{-1/2},\quad
\min_{y\in \Omega_3}|p(y,d\psi_{\hat z,\hat T}(y))|\geq 2\gamma^2,\,
\\
& &
d_{\R\times (\R^n,e)}(\Omega_2,\p \Omega_3)\geq
\frac{1}{\sqrt{b_0}} \min\{\frac{i_0}{4}, \frac{\gamma^2}{8(T-\ell)}, \frac{3\ell}{4}\}.
\eeq
Next define $R_0 = (2\sqrt{b_0})^{-1}\min\{i_0/4,\, 3\ell/4,\, \gamma^2/(8(T-\ell))\}$ a uniform radius that let the ball $B_{2R_0}(\hat y)$  stay inside the injectivity radius (in order to assure the regularity of $\psi_{\hat z, \hat T}$) and inside the set $ \Omega_3$ (in order to assure that $\psi_{\hat z, \hat T}$ is non-characteristic in the ball), according to Lemma \ref{lm_boundary}. Moreover, $\cC(\hat y,2R_0) \subset \Omega_3$.
We then consider the procedure of Appendix A to determine the radii $r,R$ related to the function $\psi = \psi_{\hat z,\hat T}(y)-\psi_{\hat z,\hat T}(\hat y)$.
We  set $R_1 = \min\{1, R_0, (\lambda|d\psi|_{C^0(\Omega_3)})^{-1}\}$ in the Table \eqref{tab1},  and we observe that, using the estimates \eqref{eq: estimat} for the derivatives of $\psi$ and Assumption A5, we can choose radii $r,R, R_2$ that are the same for each $\hat y$, and consequently also the derived parameters.
\\
As seen in section \ref{sec2}, all the parameters in Lemma \ref{lm_expo1},
$r,R$, $c_U,\,c_P,\,c_A$,
$c_{150}, c_{131}, c_{132}$, depend on the uniform estimates for the quantities listed in Proposition \ref{prop}.
As we saw above, these estimates depend on the parameters
$a_0,$ $b_0,$ $b_3,$ $T$, $\gamma$, $\ell$, and $i_0$.
Then, for each $\hat y$, the claim follows from Lemma \ref{lm_expo1} for $v$ in place of $u$, with the function $\psi = \psi_{\hat z,\hat T}(y)-
\psi_{\hat z,\hat T}(\hat y)$.
\end{proof}

\subsubsection{
Global stability estimate}

{\bf Rule of choosing the center points of small balls:} \\
We are going to apply the local stability estimate for the solution $u$ of the wave equation.
Let $r,R$ be the radii defined in
Lemma \ref{lm_expo1 application} and consider the cylinders $\cC(y_j,r/2)$  having center points at $y_j$
chosen iteratively below, see \eqref{def of C}.
For each point $y_j$ we define a localizer $b_j(y)$ associated with $y_j=(t_j,x_j)$ (see \eqref{def of b 2})
by
\ba
& &b_j(t,x)=b_{y_j,r/2}(t,x).
\ea
We proceed in analogy with Assumption A4, with $\Lambda$ in place of $\Omega_a$, with $\Lambda$ and $\Omega_0$ defined in the statement of Theorem \ref{global2 application}.
Next set $\psi(y)=\psi(y;T,z_0)$, as in \eqref{def of psi}.
The main difference is that here  $\psi$ is not everywhere a $C^{2,\rho}$ function, as explained below.
\\
We define the set ${\cal E}=\{y_j \in \overline{\Lambda}, j=1,...,J_0\}$ and the functions $u_j(y)$, iteratively as follows:
\\
1) For $j=1$ we define $u_1(t,x)=u(t,x)$ and consider $\overline{\Lambda} \subset \Omega_0$
\beq\label{yj choise j is 1}
 y_1\in \hbox{argmax}\, \{\psi(y;T,z_0)\, ; \, y=(t,x)\in \overline{\Lambda} \}.
\eeq
2)  For  $j\geq 2$, after $y_1,y_2,\dots,y_{j-1}\in {\cal E}$ have been chosen and the function $u_{j}(t,x)$ has been
 constructed we proceed as follows: If $\supp(u_{j})\cap \overline{\Lambda}\not=\emptyset$, we choose $y_j$ to be a point that satisfies
 \beq\label{yj choise}
 y_j\in \hbox{argmax}\, \{\psi(y;T,z_0)\, ;\, y=(t,x)\in \overline{\Lambda}\backslash \cup_{k=1}^{j-1} \cC(y_k,r/2)\}
 \eeq
and define
 \beq\label{u j iteration}
u_{j}(y)=(1-b_{j-1}(y))u_{j-1}(y).
\eeq
We notice that by construction $\supp(u_j) \cap \cC(y_j,2R) \subset \{y; \psi(y;T,z_0)\le \psi(y_j;T,z_0)\}$.\\
When  $\supp(u_{j})\cap \overline{\Lambda}=\emptyset$, we end the iteration and we set $J_0$ equal to $j$.
\\
Next we estimate the number of iteration steps $J_0$.
By construction, the points $y_j$ in steps 1 and 2
satisfy
$
d_{\R\times (\R^n,e)}(y_j,y_k)\geq r/2,\quad j\not =k.
$
Moreover,
\ba
\overline{\Lambda}\subset \cC(y_I,\rho_0),\quad\hbox{where } \rho_0=\frac 2{\sqrt{a_1}}(T+1),\quad y_I=(0,z_0),
\ea
see (\ref{def of C}).
Thus the maximal number $J_0$ of steps is smaller or equal to the maximal
number of points in a $\tilde r$ net in the set $\overline{\Lambda}$ that is bounded by
\beq\label{J0 estimate}
J_0\leq \frac {\hbox{vol}_{\R\times \R^n}(\cC(0,\rho_0+r))}{  \hbox{vol}_{\R\times \R^n}(\cC(0,\frac{r}{4}))}\leq C_1\frac {(T+2)^{n+1}
}{r^{n+1}}
\eeq
where $C_1$ is a uniform constant that can be estimated  in an explicit way.
\medskip
\\
Note that above we have always chosen $y_j=(t_j,x_j)$ as maximal points for the hyperbolic function $\psi(y;T,z_0)$
 associated with  the ``original'' center
 point $z_0$ and time $T$.
The motivation for this choice is that the level sets of $\psi(y;T,z_0)$ are the best approximation of the domain of influence $\Sigma(z_0,\ell,T)$ that we want to approach.
\\
When the distance of $x_j$ to the point $z_0$ is larger
 than the injectivity radius,  the function $y\mapsto \psi(y;T,z_0)$ is only Lipschitz-smooth but it may happen that
it is  not $C^2$-smooth. To apply Lemma \ref{lm_expo1 application} in this case, we choose a different
hyperbolic function $\psi_{z_j,T_j}$ that changes
at each step of the iteration and depends on the point $y_j$.

\begin{center}
\psfrag{1}{$z_0$}
\psfrag{2}{$x_j$}
\psfrag{3}{$z_j$}
\psfrag{4}{$\p B_g(z_0,d_j)$}
\psfrag{5}{\hspace{-4mm}$\p B_g(z_j,s_0)$}
\includegraphics[height=5.5cm]{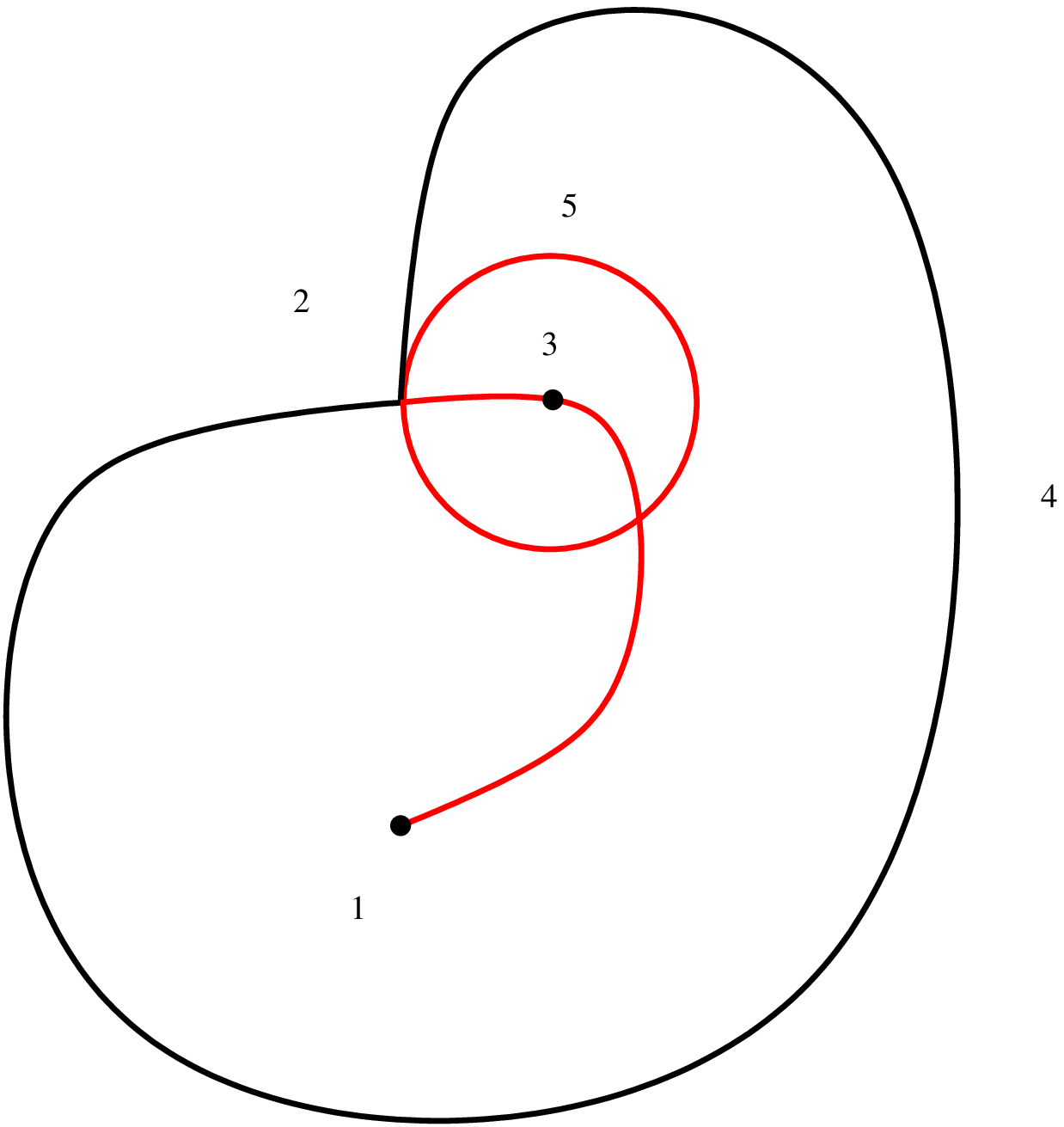}
\end{center}
{\it FIGURE 3.}  {\it Preparations to do unique continuation near $y_j=(t_j,x_j)\in \R\times \R^n$. When $d_j=d_g(z_0,x_j)$ is larger than the injectivity radius,
the boundary of the ball $B_g(z_0,d_j)$ may be non-smooth (the black external contour in the figure). For $x_j\in B_g(z_0,d_j)$
we choose some distance minimizing geodesic $\gamma_{z_0,\xi_j}([0,d_j])$ that
connects $z_0$ to $x_j$. In the figure, this geodesic is the red curve from $z_0$ to $x_j$.
On this geodesic we choose a point
$z_j=\gamma_{z_0,\xi_j}(d_j-s_0)$. The boundary of the ball $B_g(z_j,s_0)$ (the red circle
in the figure) is smooth and contains the point $x_j$. We do unique continuation near the point $y_j$
using the hyperbolic function $\psi_{z_j,T_j}$, associated with the center $z_j$, that is smooth near $y_j$.}
\medskip


We then distinguish two cases as in
Definition \ref{def-yhat}:
\\
a) If $\ell\leq d_g(z_0,x_j)\leq \frac{7i_0}{8}$, then we consider $\psi(y;T_j,z_j)$ with $z_j=z_0$ and $T_j=T$.
\\
b) Next, assume
that $d_j:=d_g(z_0,x_j)>\frac{7i_0}{8}$. Then, we define $\hat y = y_j$ and as in
Definition \ref{def-yhat}:
\ba
L_j=\hat L,\quad T_j=\hat T, \quad z_j=\hat z,\quad
\psi_{z_j,T_j}(y) = \psi(y;T_j,z_j).
\ea
Note that the choice of the point $z_j$ is not unique as there may be several distance minimizing geodesics from $z_0$
to $x_j$.

\begin{lemma}\label{lem support condition}
For the points $y_j=(t_j,x_j)\in \overline{\Lambda}$, $z_j\in \R^n$, the time $T_j>0$, and the function $u_{j}$ chosen above,
the support condition  (\ref{support condition}) is valid in the cylinder
$\cC(y_j,2R)$ for the function
$ \psi_{z_j,T_j}(y)$, that is,
\beq\label{support condition2}
\supp(u_{j})\cap \cC(y_j,2R)\subset \{y\in \R^{n+1};\ \psi_{z_j,T_j}(y)\leq \psi_{z_j,T_j}(y_j)\}.
\eeq
Moreover, we have $\psi(y_j;T,z_0)=\psi(y_j;T_j,z_j)$ and
\ba
y_j\in \p S(z_j,T_j,\ell,\gamma_j)\cap\cC(y_j,2R),\quad \hbox{where }\gamma_j:=
\sqrt{\psi(y_j;T,z_0)}\ge \gamma.
\ea
\end{lemma}
\begin{proof}
If $\ell\leq d_g(z_0,x_j)\leq \frac{7i_0}{8}$, then the property is trivial because of $\psi(y;T_j,z_j)=\psi(y;T,z_0)$ and the definition of $y_j \in {\cal E} \subset \overline{\Lambda} $.
\\
Assume now that $d_g(z_0,x_j)> \frac{7i_0}{8}$.
Recall that by definition of $R<R_0$ in the proof of Lemma \ref{lm_expo1 application}, we have
$2 R<b_0^{-1/2}\frac{\gamma^2}{T-\ell} \le b_0^{-1/2} \gamma$.
Let us consider $(t,x)\in \cC(y_j,2R)$.
By the triangle inequality and the definition of $L_j$ we have
\ba
d_g(x,z_j)+L_j\geq d_g(x,z_0),\quad
d_g(x_j,z_j)+L_j=d_g(x_j,z_0).
\ea
%
%
This yields
\beq\label{power 1}
T-(d_g(x,z_j)+L_j)\leq T-d_g(x,z_0).
\eeq
Since
\ba
d_g(x,z_j)+L_j&\leq&
d_g(x,x_j)+d_g(x_j,z_j)+L_j\;\leq\; d_g(x,x_j)+d_g(x_j,z_0)\\
&\leq & 2Rb_0^{1/2}+ T-\gamma\; \leq \; T,
\ea
we have $T-(d_g(x,z_j)+L_j)\geq 0$, and (\ref{power 1}) implies
\beq \label{EE A}
\!\!\!\!\!\!\psi(t,x;T_j,z_j)\!=\!\big(T-L_j-d_g(x,z_j)\big)^2-t^2
\leq \big(T-d_g(x,z_0)\big)^2-t^2 \!=\! \psi(t,x;T,z_0).
\eeq
Hence,
\ba
& & \supp(u_j)\cap \cC(y_j,2R)\subset
\\
& &\{(t,x)\in \cC(y_j,2R);\ \psi(t,x;T,z_0)\leq \gamma_j^2\}\subset \{(t,x)\in \cC(y_j,2R);\ \psi(t,x;T_j,z_j)\leq \gamma_j^2\}
\ea
%
Moreover, when $\tilde x\in \gamma_{z_0,\xi_j}([L_j,d_j])$, where
 $\gamma_{z_0,\xi_j}([0,d_j])$ is   a length minimizing minimizing geodesic connecting $z_0$ to $x_j$,
   and $\tilde t\in \R$, we have
 $d_g(\tilde x,z_j)+L_j=d_g(\tilde x,z_0)$ and
\beq \label{EE B}
\psi(\tilde t,\tilde  x; T_j,z_j)&=&\big(T-L_j-d_g(\tilde x,z_j)\big)^2-\tilde t\,{}^2\\ \nonumber
 &=&\big(T-d_g(\tilde x,z_0)\big)^2-\tilde t\,{}^2\; =\; \psi(\tilde t,x; T,z_0).
\eeq
In particular, when $(\tilde t,\tilde x)$ is equal to $y_j=(t_j,x_j)$, we see that
$\psi(t_j,x_j; T_j,z_j)=\psi(t_j,x_j; T,z_0)$.
%
\\
The above implies that
\beq\label{eq: implication}
\cC(y_j,2R)\cap
S(z_j,T_j,\ell,\gamma_j)
\subset
\cC(y_j,2R)\cap
S(z_0,T,\ell,\gamma_j)
.\eeq
Note that  the boundary of $S(z_0,T,\ell,\gamma_j)$
may be non-smooth in the ball $\cC(y_j,2R)$, while
the boundary of $S(z_j,T_j,\ell,\gamma_j)$
is smooth. That is why we have introduced the new function $\psi_{z_j,T_j}$.
We also recall that $\cC(y_j,2R) \subset S(z_0,T,\ell,\frac{\gamma}{\sqrt{2}})$ and $\cC(y_j,2R) \subset S(z_j,T_j,\ell,\frac{\gamma_j}{\sqrt{2}})$.\\
By the construction of $u_j$ and its support and the inclusion (\ref{eq: implication}) we deduce that
\beq
u_{j}=0,\quad \hbox{for }y\in{\cC(y_j,2R)\cap S(z_j,T_j,\ell,\gamma_j)}.
\eeq
Moreover, since $\psi(t_j,x_j; T_j,z_j)= \psi(t_j,x_j; T,z_0)=\gamma_j^2$, we have
 $y_j\in \p S(z_j,T_j,\ell,\gamma_j)\cap \cC(y_j,2R)$.
\end{proof}

\noindent{\bf Proof of  Theorem \ref{global2 application}.}
We apply Theorem \ref{global2} in a special way. As mentioned before, here Lemma \ref{lm_expo1 application} replaces Lemma \ref{lm_expo1}.\\
Set $y_j$ like in \eqref{yj choise}, and $u_j$ like in \eqref{u j iteration}.\\
Step 1. {\it Within the injectivity radius.}\quad Let $d_g(z_0,x_j) \le  7i_0/8$.
Define like in \eqref{notations 1} and Lemma \ref{lm_expo1 application}.
\ba
\Omega_{0,1}&=& \Omega_3(z_{0},T,\ell,\gamma),\\
\psi_1(y) &=& \psi(y;T,z_0),\\
\Lambda_1 &=& \Omega_2(z_0,T,\ell,\gamma_1).
\quad \gamma_1= T-\frac{5 i_0}{8} \ge \gamma.
\ea
Here $\psi_{max,1} = (T-\ell)^2$, $\psi_{min,1} = \gamma_1^2$.
Observe that $\Lambda_1 \subseteq \Omega_2(z_0,T,\ell,\gamma)$, the set used in Lemma \ref{lm_expo1 application} to compute the uniform radius $R$.
By construction every $y \in \cC(y_j, 2R)$ is such that $\frac{\ell}{4} \le d_g(z_0,x_j) \le \frac{7 i_0}{8}$, hence  $x \to d_g(z_0,x)$ is $C^3-$smooth in $\cC(y_j, 2R)$. Moreover $\psi_1$ is $C^3(\Omega_{0,1})$ and hence regular enough to apply the local stability result of Lemma \ref{lm_expo1 application}.
Here we are in the case where $\hat y = y_j$ is like
Definition \ref{def-yhat} a).
The condition \eqref{support condition} is fulfilled by $u_j$ due to the initial assumption that $u=0$ in $W(z_{0},\ell,T)$ and the construction of $u_j$ step by step. Call $N_1$ the number of points $y_j$ used to cover $\Lambda_1$.
If $T \le 7i_0/8$, then the procedure stops here. If also $T \le 5i_0/8$, then it is enough to use a fraction of $i_0$ above to define $\Omega_2$ and $\Omega_3$.
\\
Step $j > N_1$. Case $T > 7i_0/8.$\\
Here we change $\Omega_{0,j}$, $\psi_j(y)$ and $\Lambda_j$ at each step.
We have 2 cases:\\
a) If $y_j \in \Omega_{0,1}\setminus \Lambda_1$, then we simply consider $\Omega_{0,j}=\Omega_{0,1}$ and $\psi_j = \psi_1$ (that is $C^3$ regular since $d_g(x_j,z_0) \le 7i_0/8$). Here
$\Lambda_j = \Omega_2(z_0,T,\ell,\gamma_j)$, with $\gamma_j = \sqrt{\psi(y_j;T,z_0)}$, but
${\cal E}_j = \{y_j\}$,
in the sense that we apply the local unique continuation step just once, in a cylinder $\cC(y_j, 2R)$ centred in $y_j \in \{y; \psi_1(y) = \psi_1(y_j)\}$.
Observe that $\Lambda_j \subseteq \Omega_2(z_0,T,\ell,\gamma)$.
Again for $\hat y = y_j$ holds
Definition \ref{def-yhat} a) and
the condition \eqref{support condition} is fulfilled by construction. By Remark \ref{remarks} 2.b) there is no need of defining $\psi_{min,j}$ here.
\\
b) If $y_j \notin \Omega_{0,1}$. Then, $d_g(z_0,x_j) > 7 i_0/8$. \\
Here we are outside of the domain where $\psi_1$ is certainly smooth, since the function $x \to d_g(z_0,x)$ can fail to be $C^2-$smooth in $\cC(y_j, 2R)$.
So even if $y_j \in \{y; \psi_1(y) = \psi_1(y_j)\}$, to apply the local stability we choose another function $\psi_j$ passing through $y_j$ and having the good properties outlined in Lemma \ref{lem support condition}.
Calling $\hat y = y_j$ and defining $z_j,T_j,\psi_j$ as in
Definition \ref{def-yhat} b), we can consider
\ba
\Omega_{0,j} &=& \Omega_3(z_j,T_j,\ell,\gamma),\\
\psi_j(y) &=& \psi(y;T_j,z_j),\\
\Lambda_j &=& \Omega_2(z_j,T_j,\ell,\gamma_j), \quad \gamma_j = \sqrt{\psi(y_j;T,z_0)} \ge \gamma.
\ea
Observe that $\Lambda_j \subseteq \Omega_2(z_j,T_j,\ell,\gamma)$.
Again we have ${\cal E}_j = \{y_j\}$, in the sense that we apply unique continuation just in a cylinder $\cC(y_j, 2R)$ centred in $y_j \in \{y; \psi_j(y) = \psi_j(y_j)\}$.
The condition $\supp(u_j) \cap \cC(y_j,2R) \subset \{y; \psi_j(y) \le \psi_j(y_j)\}$ is fulfilled due to Lemma \ref{lem support condition}. By Remark \ref{remarks} 2.b) there is no need of defining $\psi_{min,j}$ here.
\\
Notice that, due to our uniform estimates, the radii $R$ and $r$ remain unchanged for every $y_j$ and the other constants of the Table \eqref{tab1} are chosen uniformly. This implies that $c_{156,*}=c_{156,1}$.
We also recall that $\cC(y_j,2R) \subset S(z_0,T,\ell,\frac{\gamma}{\sqrt{2}})$ for every $j$,
by the construction of the points $y_j$ and the choice of $R$.
Here $l(y) \in C^{\infty}_0(\R^{n+1})$ satisfies  $l = 1$ on $\cup_{k=1}^{J_0} \cC(y_k,2R)$, $0\le l \le 1$ and supp$(l) \subset S(z_0,T,\ell,\frac{\gamma}{\sqrt{2}})$, see Remark \ref{remarks}-4.
The coefficient $c_{163}$ is computed like $c_{161}$.
\hfill$\square$\medskip

\begin{remark}
We remark that an alternative proof of Th. \ref{global2 application} is possible by applying Th. \ref{global2} in the following way.
Define a net of center points $(t_k,z_k)$ for the translated hyperbolic functions:
$$\psi(y;T_k,z_k,t_k) = (T_k-d_g(x,z_k))^2-(t-t_k)^2.$$
Choose $\Upsilon=W(z_0,,T,\ell)$, $\Omega_{0,k} \subset \{y; y\in [-T_k +t_k, T_k+ t_k ]\times \R^n;\, \psi(y;T_k,z_k,t_k) \ge \gamma_k^2/2, \,T_k\ge d_g(x,z_k)\}$ and $\Lambda_{k} \subset \{y; y\in [-T_k +t_k, T_k+ t_k ]\times \R^n;\, \psi(y;T_k,z_k,t_k) \ge \gamma_k^2, \,T_k\ge d_g(x,z_k)\}$. The construction is similar to the one in Figure 1 of section \ref{sec2}. In this case one does not need to introduce the points $\hat y$ of Definition \ref{def-yhat}. The parameters $(t_k,z_k,T_k,\gamma_k)$ should be chosen such that $\Omega_{0,k}$ is  contained  in the domain $0 < d_g(z_k,x) \le \frac{7}{8}i_0$ (to guarantee the regularity of $\psi(y;T_k,z_k,t_k)$).
 Moreover $\Lambda_{k} \subset \Sigma(z_0,\ell,T)$ and their union should cover a subset of the domain of influence $\Sigma(z_0,\ell,T)$.
\end{remark}

\subsubsection{The case of solutions with small values in a cylinder}

Our purpose is to reformulate Theorem \ref{global2 application} for a wave equation with vanishing source term and a solution
$u$ that is no longer vanishing but it is small inside a cylindrical set.

\begin{corollary}\label{cor: source is zero}
Under the conditions of Assumption A5, let $z_0\in \R^n$.
Also, let $\Omega=(-T,T)\times B_g(z_0,T+\ell)$,
$\Omega_1= S(z_0,\ell,T,\frac{\gamma}{\sqrt{2}}) \setminus \{(t,x): t\in \R,   d_g(z_0,x) \le \ell/4 \},$
$\Omega_2= S(z_0,\ell,T, \gamma)$.
Assume that $u \in H^1(\Omega)$ satisfies
\ba
P(x,D) w(t,x)=0,\quad \hbox{for } (t,x)\in \Omega
\ea
and define
\beq\label{eq: original support modified}
W_1=(-T+\ell,T-\ell)\times B_g(z_0,\ell+\gamma).
\eeq
Then for every $0<\theta < 1$ we have
\begin{eqnarray}\label{corollary result}
\|w\|_{L^{2}(\Omega_2\setminus W_1)} \le c_{166}\frac{\|w\|_{H^1(\Omega_1\setminus W_0)}}
{\Big(\ln \Big(e + \frac{\|w\|_{H^1(\Omega_1\setminus W_0)}}{C' \|w|_{W_1}\|_{H^1(W_1)}})\Big)^\theta}\ .
\end{eqnarray}
%
%
Here, $c_{166}$ depends only on $a_0,b_0,b_3,T,
\gamma,\ell $, $i_0,$  and $\theta$.
\end{corollary}
\begin{proof}
Let $B_0=B_g(z_0,\ell)$, $B_1=B_g(z_0,\ell+\gamma)$, and $W_0=W(z_0,\ell,T)=(-T+\ell,T-\ell) \times B_0\subset W_1.$
We use a cut-off function $\eta(x)\in C^2_0(B_1)$, $0\le \eta \le 1$, that is equal
one in $B_0$ and satisfies $\|\eta\|_{C^2(\R^n)}\leq c_0\gamma^{-2}$, where $c_0$ is a uniform constant. Then
$\tilde w(x,t)=(1-\eta(x))w(x,t)$ vanishes in $W_0$ and we have
\ba
& &P(y,D) \tilde w(t,x)=F(t,x),\quad \hbox{in }\Omega,\\
& &F(t,x)=
-g^{jk}(x)(D_j \eta D_k \eta)\tilde w - g^{jk}(x)(D_j \eta D_k \tilde w)-h^j(x)(D_j \eta)\tilde w\; \in L^2(\Omega),
\ea
and since $\eta$ is supported in $B_1$,
\beq\label{F and w}
\|F\|_{L^2(\Omega_1)}\leq
\|F\|_{L^2(W_1)}
\leq c_1\|w|_{W_1}\|_{H^{1}(W_1)},
\eeq
where $c_1$  is a uniform constant.
Also, since $w=\tilde w$ in $\Omega\setminus W_1$, we have
\beq\label {v1}
\|w-\tilde w\|_{H^1(\Omega_1)}&\leq&  c_2\|w|_{W_1}\|_{H^{1}(W_1)},\\
\label {v2}
\|\tilde w\|_{H^1(\Omega_1)}&\leq& c_2 \|w\|_{H^{1}(\Omega_1)}
 \eeq
where $c_2$ is a uniform constant.
Summarizing, above we have seen that
\ba
P(y,D) \tilde w(t,x)& =&F(t,x),\quad \hbox{in }(-T+\ell,T-\ell)\times B_g(z_0,T+\ell),\\
\tilde w|_{W_0}&= &0.
\ea
Moreover $F$ in $(-T+\ell,T-\ell)\times B_g(z_0,T+\ell)
$  vanishes outside of $W_1$ and $F$ is small if $\|w|_{W_1}\|_{H^1(W_1)}$  is small.
Then, applying Theorem \ref{global2 application}  to $\tilde w$ we get
\begin{eqnarray*}
 \|\tilde w\|_{L^{2}(\Omega_2\setminus W_0)}
 \le c_{163} \frac{\|\tilde w\|_{H^{1}(\Omega_1\setminus W_0)}}{\Big(\ln\Big(e+\frac{\|\tilde w\|_{H^{1}(\Omega_1\setminus W_0)}}{\|F\|_{L^{2}(\Omega_1\setminus W_0)}}\Big)\Big)^{\theta}}.
\end{eqnarray*}
As $\|w\|_{L^{2}(\Omega_2\setminus W_1)}\leq \|\tilde w\|_{L^{2}(\Omega_2\setminus W_0)}$,
and by
(\ref{v2}),
$$
\|\tilde w\|_{H^{1}(\Omega_1\setminus W_0)}
\leq c_2 \|w\|_{H^{1}(\Omega_1\setminus W_0)} ,
$$
and
since the function
$t \mapsto \frac{t}{(\ln(e+ t))^{\theta}}$
is increasing for $t\ge 0$, we get
\begin{eqnarray*}
 \|w\|_{L^{2}(\Omega_2\setminus W_1)}
 \le
 c_{163} \frac{c_2 \|w\|_{H^{1}(\Omega_1\setminus W_0)}
 }
 {\Big(\ln\Big(e+\frac{c_2 \|w\|_{H^{1}(\Omega_1\setminus W_0)}}
 {
 c_1\|w|_{W_1}\|_{H^{1}(W_1))}
 }\Big)\Big)^{\theta}}.
 \end{eqnarray*}
This proves the claim with $c_{166}=\max(c_2/c_1,c_2c_{163})$.
\end{proof}

{\bf Acknowledgment}
RB was  supported by the postdoc programme of the Hausdorff Center for Mathematics in Bonn and by
the AXA Mittag-Leffler Fellowship Project, funded by the AXA Research Fund.
YK and RB were partly supported by EPSRC.
ML and RB were partly supported by the Academy of Finland, CoE-project 250215 and project 273979.
This work was partly done at the Isaac Newton Institute for Mathematical Sciences and the Mittag-Leffler Institute.
\\\\
When finalizing this article, it came to our attention that another group, Camille~Laurent and Matthieu~L\'{e}autaud has been working independently on issues related to this paper.

\section{Appendix}
\subsection{A: Geometric constants}

We write the table of the constants used in the article. This is a special version of Table (5.1) in \cite{B}, since now all the coefficients are calculated independently of the center point $y_k$ and of the local information.
\\
In order to get the uniform coefficients we use the same notations as in Section 3.1 of \cite{B}:
%
\begin{itemize}
\item[a)]
By Assumption A1,
we consider the case of the wave operator \eqref{wave_op} with principal symbol $p(y,\xi) = -\xi_0^2 + \sum_{jk=1}^n g^{jk}(x) \xi_j \xi_k$, with \;
$0 < a_1\, \delta^{jk} \le g^{jk}(x) \le b_1\, \delta^{jk}$, $a_1,\,b_1 >0$.\\
Call $\xi=(\xi_0,\tilde\xi)\in \R \times \R^{n}$, where $|\tilde\xi|^2 = \sum_{j=1}^n \xi_j^2$.
\item[b)]
We consider the function $\psi \in C^{2,\rho}(\R^{n+1})$, for some  $\rho\in(0,1)$,  such that $p(y,\psi'(y))\neq 0$ and $\psi'(y)\neq 0 $ in a domain $\Omega_0 \subseteq \Omega$.
Let $y_0 \in \Omega_0$ be a general point lying on the level set $S=\{y;\,\psi(y)=0\}$.
Call $p_1 = \min_{y \in \overline{\Omega}_0} p(y,\psi')>0, \,C_l= \min_{y \in \overline{\Omega}_0}{|\psi'(y)|}>0$.
\end{itemize}
Moreover we use Einstein's convention for the repeated indexes.\\
We recall the three Steps - procedure to calculate the geometric parameters in \cite{B}.
\begin{itemize}
\item[\bf Step 1].
Given a function $\psi \in C^{2,\rho}(\R^{n+1})$ fulfilling the assumptions above in a domain  $\Omega_0$,
we find positive constants $M_2, \, M_1,\, M_P$ such that the following inequality holds true
\begin{eqnarray*}\label{eq_psi}
&& M_2 \xi_0^2 + M_1\Big( \frac{|p(y,\xi+i\tau\psi'(y))|^2}{\tau^2 + |\xi|^2} + |\langle p'_\xi(y,\xi+i\tau\psi'(y),\psi'(y)\rangle|^2\Big)  \\ \nonumber
&& + \frac{ \{\overline{p(y,\xi+i\tau\psi'(y))},p(y,\xi+i\tau\psi'(y))\}}{2i \tau}
\ge M_P (\tau^2 + |\xi|^2)
\end{eqnarray*}
for every $\xi \in \R \times \R^{n},\,\xi \neq 0,\, \tau \in \R$.
The previous inequality proves that the hypersurface $S=\{y;\,\psi(y)=0\}$ is conormally strongly pseudoconvex w.r.t. $P$ in $\Omega_0$.
\item[\bf Step 2].  For $\phi=e^{\lambda \psi}$, with $y_0$ on the level set $\phi(y)=1$, we find $\lambda>0 $ such that the following inequality holds true
\begin{eqnarray*}\label{eq_phi}
&& \qquad M_2 \xi_0^2 + \frac{M_1}{\min\{1,\lambda^2 \phi^2(y)\}}  \frac{|p(y,\xi+i\tau\phi'(y))|^2}{\tau^2 + |\xi|^2}
\\ \nonumber
&& + \frac{1}{\lambda \phi(y)}\frac{ \{\overline{p(y,\xi+i\tau\phi'(y))},p(y,\xi+i\tau\phi'(y))\}}{2i \tau}
 \ge M_P \min\{1,\lambda^2 \phi^2(y)\}(\tau^2 + |\xi|^2)
\end{eqnarray*}
for every $\xi \in \R \times \R^{n},\,\xi \neq 0,\, \tau \in \R$.
The previous inequality proves that the function $\phi$ is  conormally strongly pseudoconvex w.r.t. $P$ in $\Omega_0$.
\item[\bf Step 3].  We consider a perturbation of $\phi$ by the shifted 2nd order polynomial centred in the point $y_0$,
\begin{eqnarray}\label{def_f}
f(y)= \sum_{|\upsilon|\le 2} (\partial^\upsilon {\phi})(y_0)\, (y-y_0)^\upsilon / \upsilon! - \sigma |y-y_0|^2.
\end{eqnarray}
In a ball $B(y_0,R_1) \subset \Omega_0$ where $f'\neq 0$ we define
$$\phi_0 = \min_{y \in B(y_0,R_1)} \phi(y), \qquad \phi_M = \max_{y \in B(y_0,R_1)} \phi(y). $$
We find $\sigma$ and $R_2 >0$ small enough such that in the ball $B(y_0,R_2)$  the following inequalities hold true:
$\quad f(y)< \phi(y) \quad$ in $B(y_0,\,R_2)\backslash\{y_0\}$, \quad and
\begin{eqnarray*}\label{eq_f}
&& \quad M_2 \xi_0^2 + 2 M_1  \frac{|p(y,\xi+i\tau f'(y))|^2 }{\tau^2 + |\xi|^2}
+ \frac{ \{\overline{p(y,\xi+i\tau f'(y))},p(y,\xi+i\tau f'(y))\}}{(\lambda \phi_0)2i \tau} \\
\nonumber
&&  \quad \ge \frac{1}{2}(\tau^2 + |\xi|^2)\,.\quad
\end{eqnarray*}
The previous inequality proves that the function $f$ is conormally strongly pseudoconvex w.r.t. $P$ in $B(y_0,R_2)$.
\end{itemize}
\begin{paragraph} {\bf How to obtain the uniform Table \eqref{tab1}, starting from the table in \cite{B}}.\\
We need to recalculate $R_2$ invariantly. Hence we take:
\begin{eqnarray} \label{Appendix formula1}
\sigma &=& 2 c_T R^{\rho}_2, \, c_T = n |\lambda  \psi|_{max, \Omega_0}\\
\nonumber 
|\lambda  \psi|_{max, \Omega_0} &=&
\phi_M \max(
 \lambda |\psi''|_{C^{0,\rho}(\Omega_0)},\,\lambda^2 |\psi|_{C^{0,1}(\Omega_0)}|\psi''|_{C^{0}(\Omega_0)},\,\lambda^3 |\psi|_{C^{0,1}(\Omega_0)}|\psi'|^2_{C^{0}(\Omega_0)})
\end{eqnarray}
and after a first estimate in $B_{R_1}(y_0)$ we estimate $M_R$ calculating the norms in $\Omega_0$ instead of in $B_{R_1}(y_0)$:
\begin{eqnarray*}
M_{R}\ge 1- c_{100}
\Big[
|\lambda \psi|^2_{max, {\Omega_0}} R_2^{2(1+\rho)} M_1(1+\lambda^2\phi_M^2|\psi'|^2_{C^0(\Omega_0)} )
\\
+ |\lambda \psi|_{max, {\Omega_0}} R_2^\rho \frac{1}{\lambda \phi_0}
\big(1+\lambda^2\phi_M^2|\psi'|^2_{C^0(\Omega_0)}+\lambda^2\phi_M^2(|\psi'|_{C^0(\Omega_0)}|\psi''|_{C^0(\Omega_0)} + \lambda|\psi'|^3_{C^0(\Omega_0)})\big)
\Big]\,.
\end{eqnarray*}
Recall also that
\begin{eqnarray*}
|f'|_{C^0(B_{R_2}(y_0))} \le
\lambda \phi_M |\psi'|_{C^0(\Omega_0)} + 5n |\lambda  \psi|_{max, \Omega_0} R_2^{1+\rho},
\\
|f''|_{C^0(B_{R_2}(y_0))} \le
\lambda \phi_M ( |\psi''|_{C^0} + \lambda|\psi'|_{C^0}^2) + 4n |\lambda  \psi|_{max, \Omega_0}R_2^{\rho},\\
|\phi''(y)| = |\phi \lambda (\psi'' + \lambda \psi' \times \psi')| \ge e^{-1}\lambda^2 C_l^2 /2.
\end{eqnarray*}
In the table we anticipate the definition of $\epsilon_0$ in order to embody in $R_2$ the condition $R_2 \le \frac{1}{8}\frac{\epsilon_0}{\sqrt{2M_2}}\Big(16 + \frac{1}{16}\Big)^{1/2}$, that is essential to define the inequality in Th. \ref{th_carleman}.
The estimate for the minimum $r$ was done in \cite{B} and we refine it calculating the norms in $\Omega_0$.
\\\\
Consider now the points $y_k \in {\cal E}$ defined in Assumption A4.
For each $k$, via the translation $\psi(y) - \psi(y_k)$, we have $\psi(y_k) = 0$ and we can replace $y_0$ with $y_k$ in the definitions above.
After the previous translation, i.e.
$\phi(y)=\exp(\lambda (\psi(y) - \psi(y_k)) )$,
we still have $e^{-1}= \phi_0 \le \phi(y) \le  \phi_M = e$ in each ball $B_{R_1}(y_k)$
and we can consider $\phi(y_k) = 1$. All these translated curves share the same parameters as in Table \eqref{tab1}.
\end{paragraph}
\begin{remark}\label{remark-table}
In \cite{B} we assumed that $\mbox{dist}_{\mathbb{\R}^{n+1}}\{\partial \Omega_0,\partial \Omega \}>0$.\\
Here we have the condition $\mbox{dist}_{\mathbb{\R}^{n+1}}\{\partial \Omega_0, \Omega_a \} > 0$, and then
we assume that $\Omega_a \subset \Omega_1$.
\\
In Section \ref{sec3} we apply Table \ref{tab1} to a $C^3-$function $\psi$. Instead of calculating the $C^{2,\rho}-$norm, it is more practical to use the $C^3-$norm, but this requires the following modifications. Set $\rho=1$ and $c_T$ in place of $n|\lambda \psi|_{max, {\Omega_0}}$, where
\ba
|\phi''|_{C^0(B_{R_2})} \le c_{T1}:= \lambda \phi_M (|\psi''|_{C^0(\Omega_0)}+\lambda |\psi'|^2_{C^0(\Omega_0)}),
\\
|\phi'''|_{C^0(B_{R_2})} \le c_{T2}:= \lambda \phi_M\big( 3\lambda |\psi'|_{C^0(\Omega_0)}|\psi''|_{C^0(\Omega_0)}+\lambda^2|\psi'|^3_{C^0(\Omega_0)}+|\psi'''|_{C^0(\Omega_0)}\big),\\
\sigma = 2 c_T R_2,\quad c_T = c_{T1}+ c_{T2},\; \delta = c_T q^2 \frac{R_2^3}{8},\; q= \frac{1}{4}(16 + \frac{1}{16})^{-1/2},\\
f-\phi \le -c_T q^2 R_2^2,\; |\phi'-f'|_{C^0(B_{R_2})} \le 5 c_T R_2^2,\;|\phi''-f''|_{C^0(B_{R_2})} \le 5 c_T R_2.
\ea
The bound for $r$ remains unchanged since $|\phi''|_{C^0(B_{R_2})} q^2 R_2^3/8 \le \delta$ (see \cite{B}).
\end{remark}

\vfill\newpage

\begin{center}
\begin{eqnarray}\label{tab1}
\mbox{\qquad Table for the constants calculated as in \cite{B}}
\end{eqnarray}
\begin{tabular}{p{0.8cm}|p{0.3cm}|p{12cm}}
Name &  & Limit Value \\
\hline
$C_3$ & $\ge$ & $20 (1+n^2 |g^{jk}|^2_{C^1(\Omega_0)})|\psi'|_{C^1(\Omega_0)}(1+|\psi'|_{C^0(\Omega_0)}^2)$ \\
\hline
$M_P$ & $\le$ & 1\\
\hline
$M_1$ & $\ge$ & $(M_P + C_3) \max\Big\{\frac{2}{a_1^2},\; \frac{1}{2(p_1)^2} \Big\}$ \\
\hline
$M_2$ & $\ge$ & $\frac{2}{\min\{1,a_1\}}(M_P + C_3)+ \frac{(b_1+ a_1)}{2}M_1$\\
\hline
$\lambda$ & $\ge$ & $\max\{M_1,e,\frac{2|\psi''|_{C^0(\Omega_0)}}{C_l^2}\}$ \\
\hline
$\phi_0$ & $\ge$ & $e^{-1}$ \\
\hline
$\phi_M$ & $\le$ & $e$ \\
\hline
$R_1$ & $\le$ & $\min\{1, \mbox{dist}_{\mathbb{\R}^{n+1}}\{\partial \Omega_0, \Omega_a \}, \frac{1}{\lambda|\psi'|_{C^0(\Omega_0)}}\}$ \\
\hline
$c_{100}$\quad & $\ge$ & $10(1+n^4|g^{jk}|^2_{C^1(\Omega_0)})$\\
\hline
$\epsilon_0$ & $\le$ & $\frac{1}{2n(\lambda \phi_M(|\psi''|_{C^0(\Omega_0)} + \lambda |\psi'|^2_{C^0(\Omega_0)})+4n |\lambda \psi|_{max,\Omega_0})}$\\
\hline
$R_2$ & $\le$ & $\min\Big
\{R_1,\, \Big(\frac{ C_l}{2 \phi_M (|\psi''|_{C^0(\Omega_0)} + \lambda |\psi'|^2_{C^0(\Omega_0)}  + 4n|\lambda \psi|_{max,\Omega_0}/\lambda )}\Big),\,
\Big(\frac{\lambda^2 \phi_M C_l^2
}{4n|\lambda \psi|_{max,\Omega_0}}\Big)^{\frac{1}{\rho}}
$,
\\
 &  &
$\Big(\frac{1}{4c_{100}|\lambda \psi|^2_{max,\Omega_0}  M_1(1+\lambda^2\phi_M^2|\psi'|^2_{C^0(\Omega_0)})}\Big)^{\frac{1}{2+2\rho}}$,
$\frac{1}{8}\frac{\epsilon_0}{\sqrt{2M_2}}\Big(16 + \frac{1}{16}\Big)^{1/2}$,
\\
 &  &
$\Big(\frac{\lambda \phi_0} {4c_{100}|\lambda \psi|_{max,\Omega_0}
\big(1+\lambda^2\phi_M^2|\psi'|^2_{C^0(\Omega_0)}+\lambda^2\phi_M^2(|\psi'|_{C^0(\Omega_0)}|\psi''|_{C^0(\Omega_0)} + \lambda|\psi'|^3_{C^0(\Omega_0)} \big)}\Big)^{\frac{1}{\rho}}\Big\}$\\
\hline
$\sigma$ & $\ge$  & $2n |\lambda  \psi|_{max, \Omega_0}R_2^{\rho}$\\
\hline
$\tau_0$ & $\ge$ & $\max\{1,\, 64 \Big(4 M_1 + \frac{1}{4\lambda \phi_0}\Big) \Big(
\big(\lambda \phi_M ( |\psi''|_{C^0} + \lambda|\psi'|_{C^0}^2) + 4n |\lambda  \psi|_{max, \Omega_0}R_2^{\rho}\big)^2 (1+n^2|g^{jk}|_{C^0(\Omega_0)})^2 + n|h|^2_{C^{0}(\Omega_0)}(2+2\big(\lambda \phi_M |\psi'|_{C^0(\Omega_0)} + 5n |\lambda  \psi|_{max, \Omega_0} R_2^{1+\rho}\big)^2+ 2|q|^2_{C^{0}(\Omega_0)} \Big)\}$\\
\hline
$R$ & $\le$ & $\frac{1}{4}\Big(16 + \frac{1}{16}\Big)^{-1/2} R_2$ \\
\hline
$\delta$ & $\le$ & $n \frac{1}{32}\Big(16 + \frac{1}{16}\Big)^{-1} |\lambda  \psi|_{max, \Omega_0}R_2^{2+\rho}$\\
\hline
$r$ & $\le$ & $\frac{n \lambda^2 C_l^2 \frac{1}{4}\big(16 + \frac{1}{16}\big)^{-1}R_2^{2+\rho}}{2e\big(|\phi'|_{C^{0}(\Omega_0)}+5n|\phi''|_{C^{0,\rho}(\Omega_0)}\big)}$\\
\hline
$c_{1,T}$ & $\ge$ & $ \sqrt{4 \Big(\frac{4 M_1}{\tau_0} + \frac{1}{4(\lambda \phi_0)}\Big)}$\\
\hline
$c_{2,T}$ & $\ge$ &
$
(\frac{1}{2}+ \sqrt{2M_2})(1+\frac{2|\chi_1'|_{C^0(\Omega_0)}}{\tau_0 4R}) + \frac{c_{1,T}}{\sqrt{\tau_0}}\,c_{133}
$
\\
\hline
$c_{133}$ & $\ge$ & $ 2 (1+
n^2|g^{jk}|_{C^0(\Omega_0)})\Big(\frac{|\chi_1''|_{C^0(\Omega_0)}}{\tau_0(4R)^2} + \frac{|\chi_1'|_{C^0(\Omega_0)}}{4R}(1+\lambda \phi_M |\psi'|_{C^0(\Omega_0)} + 5n |\lambda  \psi|_{max, \Omega_0} R_2^{1+\rho}+\frac{|h|_{L^\infty{(\Omega_0)}}}{\tau_0})\Big)$
\\
\hline
\end{tabular}
\end{center}
\subsubsection{Uniform estimates for the hyperbolic function $\psi$}
\begin{paragraph}
{\bf Uniform regularity estimates for the distance function $d_g$.}
\;
It is a well known fact, see \cite{DeTurck}, that if a metric is $C^m$-smooth, then the
Riemannian normal coordinates are $C^{m-1}$-smooth, and the metric tensor
in these coordinates is $C^{m-2}$-smooth. In particular, the distance function $x\mapsto d_g(x,z)$ is
$C^{m-1}$-smooth. In the following we consider how to obtain uniform
bounds for the distance function under suitable assumptions.

Let $m\geq 2$ be an integer, and $a,r_0>0$, $Q_0>0$ be fixed parameters.

Assume that on $M$ are local coordinates $(U_k,\Psi_k)$, where
 $U_k\subset M$ are open and $\Psi_k:U_k\to \R^n$ such that:
\\
\noindent
(i) For any $x\in M$ there is $k$ such that the metric balls $B_g(x,r_0)\subset U_k$,
Let $W_k=\Psi_k(U_k)\subset \R^n$.
\\
\noindent
(ii) The metric tensor satisfies in local coordinates
 \beq\label{q 4b}
\frac 14\, I\leq (\Psi_k)_*g\leq 4\, I,
\eeq
\noindent (iii) We have $\|(\Psi_k)_*g\|_{C^{m}(\overline W_k)}\leq Q_0$,
\\
 \noindent
(iv) The transition functions satisfy:\:
$
\|\Psi_k\circ \Psi_j^{-1}\|_{C^{m+1}(\Psi_j(U_j\cap U_k))}\leq Q_0,
$
\\
 \noindent
(v) The injectivity radius satisfies:\; inj$(M,g)\geq 2r_0=i_0$, where $0< i_0 < \frac{\pi}{2 \sqrt{\Lambda_M}}$ and $\Lambda_M$ is an upper bound for the sectional curvature of $M$.
\\\\
Under the previous assumptions and using the notation $\langle v \rangle=(1+|v|^2)^{1/2}$, one can obtain the estimate for the derivatives of $d_g$ when  $0 < d_g(x_0,x) < i_0$:
\beq \label{7.3}
\left|D^\alpha_x d_g(x_0, x) \right| \leq e^{c^\alpha_{m,n} \langle Q_0 \rangle^4 \langle d_g(x_0, x) \rangle^2}\,d_g(x_0, x)^{(1-|\alpha|)},
\quad |\alpha|\geq 0.
\eeq
Here $c^\alpha_{m,n}\geq 1$ are coefficients which depend only on $m, n$; their value
may be explicitly found from combinatorics.
\end{paragraph}

\noindent Consequently, we consider Assumption A5. Then for all $z\in \R^n$ and
$A_1=\{x\in \R^n;\
\frac 14 \ell\le d_g(x,z)\le \frac78 i_0\}$ we have (see \eqref{7.3} and \cite{DeTurck}
for details)
\beq\label{property 0}
& &\| d_g(\,\cdotp,z)\|_{C^1(A_1)}\leq b_2,
\quad
\| d_g(\,\cdotp,z)\|_{C^3(A_1)}\leq b_2(\ell/4)^{-2},
\eeq
where $b_2$ depends on $a_0,b_0,b_3$, and $i_0$ in an explicit way.
\\
Let $z\in M$, $T > 0$ and recall the 'hyperbolic function' introduced in Definition \ref{def3}
$$
\psi_{z,T}(t,x) :=\psi(t,x;T,z)= (T-d_g(x,z))^2 -t^2 .
$$
In the following we consider properties of this function in order to construct the related Table \eqref{tab1}.\\
{We recall that the principal symbol of the wave operator $P$, at $y=(t,x)$, $\xi=(\xi_0,\xi_1,\dots,\xi_n)\in T_y^*(\R\times \R^n)$ by
$
p(y,\xi) = -\xi_0^2+ \sum_{j,k=1}^n g^{jk}(x)  \xi_j \xi_k .
$
}

\begin{lemma}\label{lem: estimates for psi z}
 Let $z\in \R^n$,
 $A_1=\{x\in \R^n;\
\frac 14 \ell \le d_g(x,z)\le \frac78 i_0\}$ and $\mathcal A= [-T,T]\times A_1$.
 Let $y=(t,x)\in \mathcal A$
 be such that
 $|t|<T-d(x,z)$.
Also,  assume that  $\psi_{z,T}(y)\geq \gamma_I^2$, with $0 <\gamma_I < T$. Then
$d\psi_{z,T}(y)$ is well defined and satisfies
\beq\label{property 1}
& &p(y,d\psi_{z,T}(y))\geq 4\gamma_I^2,\\ \nonumber
& &\min_{y \in {\mathcal A}}| d\psi_{z,T}(y)| \geq 2\gamma_I b_0^{-1/2}. 
\eeq
Moreover, we have
\beq\label{property 2}
& &\|  \psi_{z,T}'\|_{C^0(\mathcal A)}\leq b_4(T+1) ,\nonumber
\\
& &\|  \psi_{z,T}''\|_{C^0(\mathcal A)}\leq b_4 (T+1)((\ell/4)^{-1}+1) ,\nonumber
\\
&& \|  \psi_{z,T}'''\|_{C^0(\mathcal A)}\leq b_4(T+1)((\ell/4)^{-2} +(\ell/4)^{-1}),
\eeq
where $b_4=b_4(a_0,b_0,b_3)$ is a uniform constant.
\end{lemma}

\begin{proof}
Consider the co-normal of the level set of $\psi_{z,T}$,
$$
\nu=d\psi_{z,T}(y) =(-2t\, ,\, -2 (T-d_g(x,z))\partial_{x}d_g(x,z))\in T^*_y(\R\times \R^n).
$$
Using the two following facts :
\ba
\sum_{j,k=1}^n g^{jk}(x) \partial_{x_j}d_g(x,z)\partial_{x_k}d_g(x,z)=
\|\nabla d_g(\,\cdotp,z)\|_g^2=1,
\\
\bigg|-\xi_0^2+ \sum_{j,k=1}^n g^{jk}(x)  \xi_j \xi_k\bigg|\leq
\xi_0^2+ \sum_{j,k=1}^n g^{jk}(x)  \xi_j \xi_k=:\| \xi\|_{dt^2+g}^2,
\ea
we obtain
\beq\label{property 1a}
& &p(y,d\psi_{z,T}(y)) = 4(-t^2 + (T-d_g(x,z))^2)\geq 4\gamma_I^2 > 0,\\ \nonumber
& &| d\psi_{z,T}(y)|_{dt^2+g}\geq \sqrt{4\gamma_I^2}=2\gamma_I\ \hbox{implying}\
| d\psi_{z,T}(y)|\geq 2\gamma_I b_0^{-1/2}. 
\eeq
Let us now analyze derivatives of $\psi_{z,T}$. We have
\ba
 \nabla \psi_{z,T}|_{(t,x)} &=&
 (-2t,- 2(T-d_g(x,z))\nabla_x d_g(x,z)),\\
 \nabla^2 \psi_{z,T}|_{(t,x)} &=&
 (-2,- 2(T-d_g(x,z))\nabla_x^2 d_g(x,z) +2 \nabla_x d_g(x,z )\otimes \nabla_x d_g(x,z )),
 \ea
\ba
 \nabla^3 \psi_{z,T}|_{(t,x)} &=&
  (0,V),\\
  V&=& -2(T-d_g(x,z))\nabla_x^3 d_g(x,z) + 4\nabla_x d_g(x,z )\otimes \nabla^2_x d_g(x,z )\\
  & &+2\nabla_x^2 d_g(x,z )\otimes \nabla_x d_g(x,z ).
\ea
Thus, by calling $\psi' = d\psi_{z,T}$, $\psi'' = \nabla^2 \psi_{z,T}$, $\psi''' = \nabla^3 \psi_{z,T}$, we get
\beq\label{property 2b}
\nonumber
& &\|  \psi_{z,T}'\|_{C^0(\mathcal A)}\leq 4(T+1)\| d_g(\,\cdotp,z)\|_{C^1(A_1)} ,
\\
\nonumber
& &\|  \psi_{z,T}''\|_{C^0(\mathcal A)}\leq 4(T+1)(\| d_g(\,\cdotp,z)\|_{C^2(A_1)} + \| d_g(\,\cdotp,z)\|^2_{C^1(A_1)}) ,
\\
&& \|  \psi_{z,T}'''\|_{C^0(\mathcal A)}\leq 6(1+T)(\| d_g(\,\cdotp,z)\|_{C^3(A_1)} + \| d_g(\,\cdotp,z)\|_{C^1(A_1)}\| d_g(\,\cdotp,z)\|_{C^2(A_1)}),\qquad
\eeq
and then one can use \eqref{property 0} at the right hand side, where $b_4$ is a uniform constant.
\end{proof}

Next we estimate the distance between two level sets of $\psi_{z,T}$, or $\psi_{\hat z,\hat T}$, outside of a cylinder of radius $\ell$
(see Definition \ref{def3}) and its consequences.

\begin{lemma} \label{lm_boundary}
Under the Assumption A5, we have,\\
a) \; For $i=1,2$ define $L_i = \{y; y \in \p S(z,T,\ell,\gamma_i), d_g(x,z) > \ell\}$, with $\gamma_1 = \gamma/\sqrt{2}$ and $\gamma_2 = \gamma$. Hence,
\beq\label{property x}
\dist_{dt^2+g}(L_1,L_2)
\geq
 \frac{\gamma^2}{8(T-\ell)}:= c_{180}.
\eeq
Consequently, defining $z_0$, $\Lambda$ and $\Omega_0$ as in Theorem \ref{global2 application}
we get
\beq\label{property x1}
\dist_{\R^{n+1}}(\Lambda, \p \Omega_0) \ge \frac{1}{\sqrt{b_0}} \min\{c_{180}, \frac{3\ell}{4}\}.
\eeq
b) \;
Let $\hat y = (\hat x, \hat t)$ be a point defined in
Definition \ref{def-yhat}, and let $\psi(y;\hat T,\hat z)$ be the associated hyperbolic function. For $i=3,4$ define $L_i = \{y; y \in \p S(\hat z,\hat T,\ell,\gamma_i), d_g(x,\hat z) > \ell\}$.\,
Then, for $\gamma_3 = \gamma /\sqrt{2}$ and $\gamma_4 = \gamma$,
\beq\label{property x2}
\dist_{dt^2+g}(L_3,L_4) \geq c_{180}.\;
\eeq
\end{lemma}
\begin{proof} a) Let $y_0=(t_0,x_0) \in L_2$
and  $y_1=(t_1,x_1) \in L_1$.
Our aim is to get a positive lower bound for $\dist_{dt^2+g}(y_1,y_0)$.
The two points lay on two level sets of $\psi(y;T,z)$, and we consider their difference:
\beq\label{d-g-01}
\psi(y_0;T,z) - \psi(y_1;T,z) = (T-d_g(z,x_0))^2-t_0^2 - (T-d_g(z,x_1))^2+t_1^2 =  \frac{\gamma^2 }{2}
\nonumber\\
(d_g(z,x_1)-d_g(z,x_0))(2T-d_g(z,x_0)-d_g(z,x_1))+t_1^2-t_0^2 =  \frac{\gamma^2 }{2}.
\eeq
By definition, we know that
$$(2T-d_g(z,x_0)-d_g(z,x_1)) \ge 0,\;\;
\ell \le d_g(z,x_0)\le  T - \gamma,\;\; \ell \le d_g(z,x_1)\le T- \frac{\gamma}{\sqrt{2}} < T.$$
Assume w.r.o.g. that $t_0,t_1 \ge 0$.\\
Case 1. $t_1^2-t_0^2 =  m\frac{\gamma^2 }{2}$, $m \in [0,1)$,
\ba
d_g(x_1,x_0)\ge (d_g(z,x_1)-d_g(z,x_0))= \frac{(1-m)\gamma^2 }{2(2T-d_g(z,x_0)-d_g(z,x_1))}\ge \frac{(1-m)\gamma^2 }{2(2T-2\ell)}.
\ea
Case 2. $t_1^2-t_0^2 =  \frac{\gamma^2 }{2}$,
that implies $d_g(z,x_0)=d_g(z,x_1)$ and $d_g(x_1,x_0)\ge 0$.
\\
Case 3. $t_1^2-t_0^2 = q \frac{\gamma^2 }{2}$, $q >1$
\ba
(d_g(z,x_1)-d_g(z,x_0))(2T-d_g(z,x_0)-d_g(z,x_1)) =  -(q-1)\frac{\gamma^2 }{2}
\ea
here one can reverse the signs and prove as in case 1:
\ba
d_g(x_1,x_0)\ge (d_g(z,x_0)-d_g(z,x_1))= \frac{(q-1)\gamma^2 }{2(2T-d_g(z,x_0)-d_g(z,x_1))}\ge \frac{(q-1)\gamma^2 }{2(2T-2\ell)}.
\ea
Case 4. $t_1^2-t_0^2 = -p\frac{\gamma^2 }{2}< 0$, for $p>0$\\
\ba
(d_g(z,x_1)-d_g(z,x_0))(2T-d_g(z,x_0)-d_g(z,x_1)) =  \frac{(1+p)\gamma^2 }{2}.
\ea
Case 1. is then the worse case. Hence for $t_1 = \sqrt{m\frac{\gamma^2}{2}+ t_0^2} \le \sqrt{\frac{\gamma^2}{2}+ (T-\ell)^2}$ and since $\gamma \le (T-\ell)$, we have
\ba
\dist_g(y_1,y_0)= \max\{|t_1-t_0|=\frac{(t_1^2-t_0^2)}{(t_1+t_0)}, d_g(x_1,x_2)\} \ge \\
\ge
\max\{m\frac{\gamma^2}{(\sqrt{\frac{(T-\ell)^2}{2}+ (T-\ell)^2}+T-\ell)},
(1-m)\frac{\gamma^2 }{2(2T-2\ell)}\}
\ge \frac{\gamma^2}{8(T-\ell)}.
\ea
b) We then consider $y_0,y_1$ belonging to two level sets of the function $\psi(y;\hat z,\hat T)$ to calculate the left hand site of \eqref{property x2}. We repeat the same computation as above with the new values.
We recall \eqref{property x5}.
By triangular inequality, since $\hat T = T -d_g(z_0,\hat z)$,
$$(2\hat T-d_g(\hat z,x_0)-d_g(\hat z,x_1)) \le 2T - d_g(z_0,x_0)-d_g(z_0,x_1) \le 2(T - \ell).$$
And again $t_1 = \sqrt{m\frac{\gamma^2}{2}+ t_0^2} \le \sqrt{\frac{(\hat T-\ell)^2}{2}+ (\hat T-\ell)^2}\le 2(\hat T-\ell)\le 2(T-\ell)$.\\
Hence the new estimate is bounded from below by $c_{180}$.
\end{proof}

{
\noindent \small
\\
Roberta Bosi\\
Department of Mathematics and Statistics, University of Helsinki\\
P.O. Box 68,
FI-00014 Helsinki.
\\
E-mail address: roberta.bosi@helsinki.fi
\\\\
Yaroslav Kurylev\\
Department of Mathematics, University College London\\
Gower Street, LONDON WC1E 6BT
\\
E-mail address: y.kurylev@ucl.ac.uk
\\\\
Matti Lassas\\
Department of Mathematics and Statistics, University of Helsinki\\
P.O. Box 68,
FI-00014 Helsinki.
\\
E-mail address: matti.lassas@helsinki.fi
}
\end{document}